\documentclass[10pt,a4paper]{article}
\usepackage[margin=1in]{geometry}
\usepackage[T1]{fontenc}
\usepackage[utf8]{inputenc}
\usepackage{lmodern}
\usepackage{amsmath,amssymb,amsfonts,amsthm,amsopn}
\usepackage{graphicx}
\usepackage{tikz}
\usepackage{enumitem}

\usepackage{xcolor}
\usepackage[colorlinks=true,citecolor=citesky,linkcolor=red,pdfpagetransition=Blinds,backref=page]{hyperref}
\usepackage{xcolor}
\definecolor{citesky}{RGB}{85,165,215}
\usepackage[capitalise,nameinlink]{cleveref}
\numberwithin{equation}{section}
\usepackage{orcidlink}

\setlist[enumerate]{leftmargin=.5in}
\setlist[itemize]{leftmargin=.5in}

\usepackage{amsthm}
\usepackage{aliascnt}

\newtheorem{theorem}{Theorem}[section]

\newaliascnt{proposition}{theorem}
\newtheorem{proposition}[proposition]{Proposition}
\aliascntresetthe{proposition}

\newaliascnt{lemma}{theorem}
\newtheorem{lemma}[lemma]{Lemma}
\aliascntresetthe{lemma}

\newaliascnt{corollary}{theorem}
\newtheorem{corollary}[corollary]{Corollary}
\aliascntresetthe{corollary}

\newaliascnt{remark}{theorem}
\newtheorem{remark}[remark]{Remark}
\aliascntresetthe{remark}

\usepackage{hyperref}
\usepackage{aliascnt}
\usepackage[capitalise,nameinlink]{cleveref}
\newtheorem{definition}[theorem]{Definition}

\crefname{theorem}{theorem}{theorems}
\crefname{proposition}{proposition}{propositions}
\crefname{lemma}{lemma}{lemmas}
\crefname{corollary}{corollary}{corollaries}
\crefname{remark}{remark}{remarks}

\usepackage{hyperref}
\usepackage{aliascnt}
\usepackage[capitalise,nameinlink]{cleveref}

\newcommand{\oy}{\Omega_y}
\newcommand{\N}{\mathbb{N}}
\newcommand{\norm}[1]{\left\|#1\right\|}
\newcommand{\rea}{\mathbb{R}}
\newcommand{\pa}{\partial}
\newcommand{\Ltwo}{L^2(\Omega)}
\newcommand{\LtwoGamma}{L^2(\Gamma)}
\newcommand{\Hsp}{\mathcal{H}(\Omega)}
\newcommand{\Hdual}{\mathcal{H}'(\Omega)}
\newcommand{\ox}{\Omega_x}
\newcommand{\ens}[1]{\left\{#1\right\}}
\newcommand{\ip}[2]{\left\langle #1,#2\right\rangle}
\newcommand{\tl}{\tilde}
\newcommand{\mc}{\mathcal{}}
\newcommand{\F}{\mathcal S}
\newcommand{\Y}{\mathcal{Y}}
\newcommand{\V}{\mathcal{V}}

\title{On the Controllability of the Kuramoto--Sivashinsky Equation
in Multi-Dimensional Cylindrical Domains}
\author{V\'ictor Hern\'andez-Santamar\'ia\thanks{Departamento de Matem\'aticas,
Facultad de Ciencias, Universidad Nacional Aut\'onoma de M\'exico,
Circuito Exterior, C.U., 04510, Coyoac\'an, CDMX, Mexico.
\texttt{victor.santamaria@ciencias.unam.mx}.} \,\orcidlink{0000-0002-2294-1376}
\and Subrata Majumdar \thanks{Instituto de Matem\'aticas,
Universidad Nacional Aut\'onoma de M\'exico,
Circuito Exterior, C.U., 04510, Coyoac\'an, CDMX, Mexico.
\texttt{subratamajumdar634@gmail.com}.}\, \orcidlink{0000-0001-6724-6943}}
\date{}

\begin{document}

\maketitle

\begin{abstract}
In this article, we investigate the null controllability of the Kuramoto-Sivashinsky (KS) equation on a cylindrical domain $\Omega=\Omega_x\times \Omega_y$  in $\mathbb R^N$, where $\Omega_x=(0,a)$ with $a>0$ and $\Omega_y$ is a smooth domain in $\mathbb R^{N-1}$. We first study the boundary controllability of this  system by a control acting on $\{0\}\times \omega$, $\omega\subset \Omega_y$ via the boundary term associated with the Laplacian component. The null controllability of the linearized system is established using a combination of two techniques: the method of moments and the Lebeau-Robbiano strategy. We provide a necessary and sufficient condition for the null controllability of this system, along with an explicit control cost estimate. {Furthermore, we show that there exists a threshold $T_0(x_0)\in [0,+\infty]$ governed by the Diophantine behavior of $x_0/a\in (0,1)\setminus \mathbb{Q}$ such that the system is null controllable for all time $T > T_0(x_0)$ by means of an interior control exerted on $\gamma = \{x_0\} \times \omega \subset \Omega$, and fails to be controllable when if $T<T_0(x_0).$}
		Assuming that $x_0/a$ is an algebraic real number of order $d > 1$, we prove null controllability for any time $T>0.$
		  Finally, for $N\in \{2, 3\}$, we establish the local null controllability of the main nonlinear system by employing the source term method combined with the Banach fixed-point theorem. 
\end{abstract}

\par

\par

	\setcounter{tocdepth}{2}
\begingroup
\hypersetup{linkcolor=black}
\tableofcontents
\endgroup

\section{Introduction and main results}
\subsection{Motivation}

The Kuramoto–Sivashinsky (KS) equation is a paradigmatic model in the study of nonlinear partial differential equations. Originally introduced in the 1970s by Kuramoto, Tsuzuki, and Sivashinsky, it models a range of physical phenomena including crystal growth \cite{KT75,KT76} and flame front instabilities \cite{Siv77}. The scalar form of the equation reads
\begin{equation}\label{eq:kse}
\partial_t u + \Delta^2 u + \nu \Delta u + \frac{1}{2} |\nabla u|^2 = 0,
\end{equation}
and is typically supplemented with appropriate boundary conditions (e.g., Dirichlet, Navier, or periodic) and an initial condition. Here, $\nu$ is a positive real parameter.

Over the past four decades, the KS equation has been the subject of extensive analytical study. In the one-dimensional setting, the equation is known to be globally well-posed (see, e.g., \cite{Tad86}), and many qualitative and quantitative properties have been established \cite{HN85,NST85,Gru00,Goo94,GO05,Ott09,KTZ18}. In contrast, the analysis of \eqref{eq:kse} in two or more spatial dimensions remains challenging. A general global well-posedness result remains lacking, and only partial results are available. For example, local well-posedness in $L^p$ spaces has been investigated in \cite{BS07,IS16}, while global results are known only under restrictive assumptions, such as thin domains, anisotropic reductions, or for modified versions of the equation \cite{ST92,BKR14,LK20,KM23,CDF21,FM22,ELW24}.

Motivated by the importance of understanding qualitative properties of higher-dimensional models, the goal of this paper is to study the controllability of the KS equation in spatial dimensions two and higher. Our contributions are twofold: on the one hand, we extend known results from the one-dimensional setting to higher dimensions; on the other hand, we improve and generalize some of the (few) existing controllability results available in dimensions two or more.

\subsection{Problem formulation}

To set the framework, let us consider equation \eqref{eq:kse} posed on a bounded domain $\Omega \subset \mathbb{R}^N$ with $N \geq 2$ (to be specified later), and impose Navier-type boundary conditions. In more detail, let $T>0$ and consider the system
\begin{equation}\label{NKS_cyl}
	\begin{cases}
		\partial_t u+\Delta^2 u+\nu\Delta u+\frac{1}{2}|\nabla u|^2=0 &  \text{ in } (0,T)\times \Omega,\\
		u=0, \quad \Delta u=\mathbf{1}_{\Gamma}q & \text{ on } (0,T)\times \partial \Omega, \\
		u(0)=u_0 & \text{ in } \Omega,
	\end{cases}
\end{equation}
where $\Gamma$ is a nonempty, open subset of $\partial \Omega$, $\mathbf{1}_{\Gamma}$ denotes its indicator function, and $u_0\in L^2(\Omega)$ is a given initial data. The function $q=q(t,x)$ will act as a control applied only on the subset $\Gamma$.

The classical objective in controllability theory is to determine whether there exists a control $q$ such that the solution of system \eqref{NKS_cyl} reaches a desired final state. In our case, we are interested in \emph{null controllability}, that is, whether one can steer the solution to rest at a given final time $T>0$. In other words, if there is a control $q$ such that $u(T,\cdot)=0$ in $\Omega$. 

Let us consider domains $\Omega$ of the form $\Omega=\Omega_x\times \Omega_y$ in $\mathbb R^{N}$, $N\geq2$, where $\Omega_x=(0,a)$, and $\Omega_y$ is a bounded smooth domain in $\mathbb R^{N-1}$. Let $\Gamma=\{0\}\times \omega$ with $\omega\subset \Omega_y$ an open subset (see \Cref{fig:cyl_domain}). 

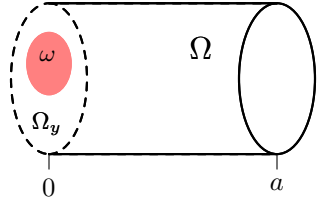
\begin{figure}[htbp!]
	\centering
	\begin{tikzpicture}[line cap=round,line join=round,scale=1]
		\def\a{1} 
		\def\h{0} 
		\def\z{3} 
		
		\foreach\i in {0,\h} 
		{%
			\draw[thick,dashed] (\z,\i+\a) -- (0,\i+\a) arc (90:270:0.5*\a cm and \a cm) -- (\z,\i-\a) ;
			\draw[thick, dashed] (0,\i-\a) arc (-90:90:0.5*\a cm and \a cm);
			\draw[thick] (\z,\i) ellipse (0.5*\a cm and \a cm);
			\draw[thick] (\z,\i+\a) -- (0,\i+\a);
			
			\draw[fill=red!50,draw=none] (0,0.2) ellipse (0.30*\a cm and 0.42\a cm);
			
			\draw[thick] (0,-1) -- (3,-1);
			\node at (2.0, 0.4) {\Large$\Omega$};
			\node[right] at (-0.25,0.3) {\small $\omega$};
			\node[right] at (-0.35,-0.6) {\small $\Omega_y$}; 
		}
		\draw[] (0,-\a) -- (0,-1.2) node [below] {$0$};
		\draw[] (\z,-\a) -- (\z,-1.2) node [below] {$a$};
	\end{tikzpicture}  
	\caption{Example of the domain $\Omega$ for equation \eqref{NKS_cyl} with $N=3$. The red region, denoted by $\omega$, represents the control set.}
	\label{fig:cyl_domain}
\end{figure}

To state our results, we begin by fixing some notation. Let
$\{\mu_j^{\Omega_y},\Psi_j^{\Omega_y}\}$ denote the eigenpairs of the
following Dirichlet Laplacian eigenvalue problem
\begin{equation}\label{eg_value}
-\Delta \Psi = \mu \Psi \;\;\text{in } \Omega_y,
\quad \quad
\Psi=0 \;\;\text{on } \partial \Omega_y.
\end{equation}
Based on these eigenvalues, we define the following countable set of
critical values
\begin{equation}\label{critical value_cyl1}
	\mathcal{N}
	=
	\bigg\{
	2\mu^{\oy}_j+\pi^2\left(\frac{k^2+l^2}{a^2}\right):
	k,l,j \in \N,\;\; k\neq l
	\bigg\}.
\end{equation}
We are now in a position to state our main results.
Let us first consider the following linearized equation of
\eqref{NKS_cyl}
\begin{equation}\label{lin_NKS_cyl}
	\begin{cases}
		\partial_t u+\Delta^2 u+\nu\Delta u=0
		& \text{in } (0,T)\times \Omega,\\
		u=0, \quad \Delta u=\mathbf{1}_{\Gamma}q
		& \text{on } (0,T)\times \partial \Omega, \\
		u(0)=u_0
		& \text{in } \Omega.
	\end{cases}
\end{equation}

\begin{theorem}\label{Ln_KSE}
	Assume that $N\geq 2$ and $T>0$. For any $u_0\in L^2(\Omega)$,
	there is $q \in L^2(0,T;L^2(\Gamma))$ such that system
	\eqref{lin_NKS_cyl} satisfies $u(T,\cdot)=0$ in $\Omega$ if and
	only if $\nu\notin\mathcal N$. Moreover, the control satisfies
	the following estimate:
	\begin{equation*}
	\|q\|_{L^2(0,T;L^2(\Gamma))}
	\leq Ce^{\frac{C}{T^{1/3}}}\|u_0\|_{L^2(\Omega)},
	\end{equation*}
	for some positive constant $C$ independent of $T$ and $u_0$.
\end{theorem}

{The proof of this result is developed throughout \Cref{sec:cylinder}, where several intermediate results are established.}

From \Cref{Ln_KSE} and a general method to deal with the controllability
of nonlinear parabolic systems, we deduce the local null controllability
of system \eqref{NKS_cyl}. The result is the following.

\begin{theorem}\label{th_nl_KS}
	Assume that $N=2$ or $3$. If $\nu\notin\mathcal N$, then the KS
	equation \eqref{NKS_cyl} is locally null controllable for any
	$T>0$, that is, there exists $R>0$ such that for any
	$u_0\in L^2(\Omega)$ satisfying
	$\norm{u_0}_{L^2(\Omega)}\leq R$, there is
	$q \in L^2(0,T;L^2(\Gamma))$ such that system \eqref{NKS_cyl}
	satisfies $u(T,\cdot)=0$ in $\Omega$.
\end{theorem}

{The proof of \Cref{th_nl_KS} is given in \Cref{sec:nonlinear}.}

\begin{remark}
Note that our main result for the linear system holds under a necessary
and sufficient condition on the parameter $\nu$, whereas for the
nonlinear problem, this condition is only sufficient. Nevertheless, the
nonlinear equation \eqref{NKS_cyl} may still be null controllable even
when $\nu\in\mathcal N$, which presents an interesting question. For a
similar problem in the case of the Korteweg--de Vries equation, see
\cite{C07} and \cite{CC09}.
\end{remark}

We can also extend our control results in a different setting, more precisely when the control acts in the interior of the domain. Let $\Omega$ be the domain as defined above and we assume $\gamma=\{x_0\}\times \omega,$ where $x_0\in (0,a)$, $\omega\subset \Omega_y,$ and we consider the following interior control problem
\begin{equation}\label{int}
	\begin{cases}
		\partial_t u+\Delta^2 u+\nu\Delta u=\delta_{x_0}\mathbf{1}_{\omega}(y) h(t,y) &  \text{ in } (0,T)\times \Omega,\\
		u=0, \quad \Delta u=0 & \text{ on } (0,T)\times \partial \Omega, \\
		u(0)=u_0 & \text{ in } \Omega.
	\end{cases}
\end{equation}
We prove the controllability of system \eqref{int} with $u_0\in L^2(\Omega)$ by means of a control  $h \in L^2(0,T;L^2(\Omega_y))$ for two cases, when $\omega=\Omega_y$ and $\omega\subsetneq \Omega_y,$ see \Cref{fig:cyl_domain4} and \Cref{fig:cyl_domain3}.
\begin{figure}[htbp!]
	\centering
	\begin{tikzpicture}[line cap=round,line join=round,scale=1]
		\def\a{1} 
		\def\h{0} 
		\def\z{3}
		\def\n{2} 
		
		\foreach\i in {0,\h} 
		{%
			\draw[thick,dashed] (\z,\i+\a) -- (0,\i+\a) arc (90:270:0.5*\a cm and \a cm) -- (\z,\i-\a) ;
			\draw[thick, dashed] (0,\i-\a) arc (-90:90:0.5*\a cm and \a cm);
			\draw[thick] (\z,\i) ellipse (0.5*\a cm and \a cm);
			\draw[thick] (\z,\i+\a) -- (0,\i+\a);
			
			\draw[fill=red!50,draw=none] (1.5,0.0) ellipse (0.30*\a cm and \a cm);
			
			\draw[thick] (0,-1) -- (3,-1);
			\node at (2.0, 0.7) {\Large$\Omega$};
			\node[right] at (-0.35,-0.6) {\small $\Omega_y$}; 
		}

		\draw[] (0,-\a) -- (0,-1.2) node [below] {$0$};
		\draw[] (\z,-\a) -- (\z,-1.2) node [below] {$a$};
		\fill[black] (1.5,-1) circle (1pt) node[below] {$x_0$};
	\end{tikzpicture}  
	\caption{ $\omega=\Omega_y$.}
	\label{fig:cyl_domain4}
\end{figure}
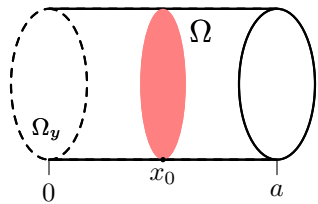
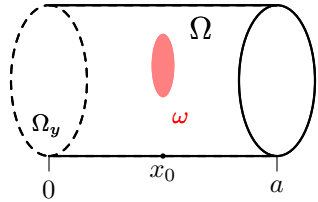
\begin{figure}[htbp!]
	\centering
	\begin{tikzpicture}[line cap=round,line join=round,scale=1]
		\def\a{1} 
		\def\h{0} 
		\def\z{3}
		\def\n{2} 
		
		\foreach\i in {0,\h} 
		{%
			\draw[thick,dashed] (\z,\i+\a) -- (0,\i+\a) arc (90:270:0.5*\a cm and \a cm) -- (\z,\i-\a) ;
			\draw[thick, dashed] (0,\i-\a) arc (-90:90:0.5*\a cm and \a cm);
			\draw[thick] (\z,\i) ellipse (0.5*\a cm and \a cm);
			\draw[thick] (\z,\i+\a) -- (0,\i+\a);
			
			\draw[fill=red!50,draw=none] (1.5,0.2) ellipse (0.15*\a cm and 0.42\a cm);
			
			\draw[thick] (0,-1) -- (3,-1);
			\node at (2.0, 0.7) {\Large$\Omega$};
			\node[right] at (1.5,-0.5) {\small ${\color{red}\omega}$};
			\node[right] at (-0.35,-0.6) {\small $\Omega_y$}; 
		}

		\draw[] (0,-\a) -- (0,-1.2) node [below] {$0$};
		\draw[] (\z,-\a) -- (\z,-1.2) node [below] {$a$};
		\fill[black] (1.5,-1) circle (1pt) node[below] {$x_0$};
	\end{tikzpicture}  
	\caption{$\omega\subsetneq\Omega_y$}
	\label{fig:cyl_domain3}
\end{figure}
 Let us state the following results:
\begin{theorem}\label{thm_int}
	Let $T>0$, $u_0\in L^2(\Omega)$, and $\nu\notin\mathcal{N}$. Also assume that $x_0/a\in(0,1)\setminus\mathbb{Q}$. Then, we have the following:
	\begin{itemize}
		\item[1.] \label{full}
		Let $\omega=\Omega_y$. Denote
		\begin{align}\label{T_0}
			[0,+\infty]\ni T_0(x_0):=\limsup\limits_{k\to+\infty}\frac{-\log\left(\left|
			\sin\left(\frac{k\pi x_0}{a}\right)\right|\right)}{\frac{k^4\pi^4}{a^4}}.
		\end{align}
		Then system \eqref{int} is null controllable in time
		$T>T_0(x_0)$ and not null controllable in time
		$T<T_0(x_0)$.

		\item[2.] \label{local}
		Let $x_0/a$ be an algebraic real number of order $d>1$ and
		$\omega\subsetneq\Omega_y$. Then system \eqref{int} is null
		controllable for all $T>0$.
	\end{itemize}
\end{theorem}

{ The proof of \Cref{thm_int} is presented in \Cref{sec:internal}, where it is developed through a series of intermediate results.}

\begin{remark}
For case 1, the conditions $\nu\notin\mathcal{N}$ and
$x_0/a\in(0,1)\setminus\mathbb{Q}$ are also necessary to have null
controllability at time $T>T_0(x_0)$; see \Cref{nec_int} and the proof
of \Cref{thm_int} for details.
\end{remark}

\subsection{Bibliographic comments on control problems for the KS equation in higher dimensions}\label{sec_tak}
Let us discuss some of the most relevant control problems for fourth-order parabolic equations in dimension $N \geq 2$. In \cite{GK19}, the authors considered the following problem
\begin{equation*}
	\begin{cases}
		\partial_t u+\Delta^2 u=\mathbf{1}_{\omega}v &  \text{ in } (0,T)\times \Omega,\\
		u=0, \quad \Delta u=0 & \text{ on } (0,T)\times \partial \Omega, \\
		u(0)=u_0 & \text{ in } \Omega,
	\end{cases}
\end{equation*}
where $\Omega$ is a smooth bounded domain in $\rea^N$.
They studied Carleman estimates for the corresponding adjoint system, using an arbitrarily small open set $\omega$ as the observation region. This analysis led to an interior null controllability result. By employing classical results (see pages 28–29 of \cite{FI96}) and extending the domain $\Omega$, one can also obtain a boundary controllability result with two boundary controls for the following system 
\begin{equation*}
	\begin{cases}
		\partial_t u+\Delta^2 u=0 &  \text{ in } (0,T)\times \Omega,\\
		u=\mathbf{1}_{\Gamma}q_1, \quad \Delta u=\mathbf{1}_{\Gamma}q_2 & \text{ on } (0,T)\times \partial \Omega, \\
		u(0)=u_0 & \text{ in } \Omega,
	\end{cases}
\end{equation*}
where $\Gamma$ is an open subset of the boundary $\pa \Omega.$ Whether it is possible to achieve controllability with fewer boundary controls remains an open question in the general case of smooth domains in multiple dimensions.

Let us also mention the related work \cite{LRR20}, where the authors established a spectral inequality for the bi-Laplace operator with Dirichlet boundary conditions, relying on several Carleman estimates. This inequality is then used to directly deduce the internal null controllability of the system
\begin{equation*}
	\begin{cases}
		\partial_t u+\Delta^2 u=\mathbf{1}_{\omega}v &  \text{ in } (0,T)\times \Omega,\\
		u=0, \quad \dfrac{\pa u}{\pa \nu}=0 & \text{ on } (0,T)\times \partial \Omega, \\
		u(0)=u_0 & \text{ in } \Omega.
	\end{cases}
\end{equation*}
The result in the literature most closely related to our study is presented in \cite{TT17} and provides the main motivation for our analysis. In that work, the author studied the boundary null controllability of the two-dimensional Kuramoto--Sivashinsky equation \eqref{NKS_cyl} on a rectangular domain $\Omega=(0,a)\times(0,b)$. Local null controllability of \eqref{NKS_cyl} was proved under the following assumptions:
\begin{itemize}[noitemsep, topsep=0.5mm, leftmargin=*]
	\item $\nu\in(0,\sqrt{\lambda_1})$, where $\lambda_1$ is the smallest eigenvalue of the bi-Laplace operator 
	(see \Cref{sec:framework} for further details)
	\[
	A_0 : \left\{ u \in H^4(\Omega) \; ; \; u = \Delta u = 0 \text{ on } \partial \Omega \right\} \longrightarrow L^2(\Omega), \quad 
	u \mapsto \Delta^2 u.
	\]

	\item $\Gamma$ contains both a horizontal and a vertical segment of positive length as indicated in \Cref{fig:region}. 
\end{itemize}

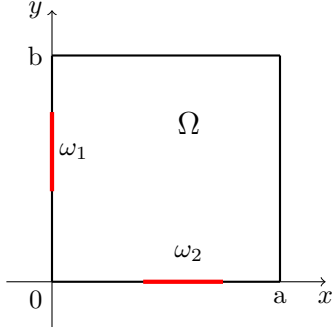
\begin{figure}[htbp!]
	\centering
	\begin{tikzpicture}[scale=3] 
		\draw[thin, -] (-0.2,0) -- (0,0); 
		\draw[thin, ->] (1,0) -- (1.2,0) node[below] {$x$}; 
		
		\draw[thin, -] (0,-0.2) -- (0,0); 
		\draw[thin, ->] (0,1) -- (0,1.2) node[left] {$y$}; 
		
		\draw[ thick] (0,0) -- (1,0) ;
		\draw[ thick] (1,0) -- (1,1) ;
		\draw[ thick] (1,1) -- (0,1) ;
		\draw[ thick] (0,1) -- (0,0);
		
		\draw[ultra thick, red] (0,0.4) -- (0,0.75);
		\draw[ultra thick, red] (0.4,0) -- (0.75,0); 
		\node[left] at (0.2,0.575) {$\omega_1$}; 
		\node[below] at (0.6,0.2) {$\omega_2$};
		\node[below left] at (0,0) {0};
		\node[below] at (1,0) {a};
		\node[left] at (0,1) {b};
		\node at (0.6, 0.7) {\Large$\Omega$}; 
	\end{tikzpicture}
	\caption{The domain $\Omega$ for equation \eqref{NKS_cyl}, with the control region $\omega=\omega_1\times\omega_2$ shown in red.}
	\label{fig:region}
\end{figure}
The proof presented in \cite{TT17} relies on general results concerning the observability
and controllability of parabolic and hyperbolic systems. The method can be summarized as follows: first, using a classical technique due to Russell \cite{R73}, one deduces the null controllability (or observability) of the linearized equation from the controllability (or observability) of an associated hyperbolic equation. Then, the linearized system with a source term is controlled using a result from
\cite{Tucsnak-nonlinear}. Finally, a fixed-point argument is applied to obtain the desired result for the nonlinear equation.

In particular, \Cref{th_nl_KS} improves upon the main result of \cite{TT17} in several ways when $N=2$, that is, when $\Omega_y = (0,b)$ for some $b>0$:
\begin{itemize}[noitemsep, topsep=0.5mm, leftmargin=*]
	\item We replace the sufficient condition $\nu\in(0,\sqrt{\lambda_1})$ by a more general one $\nu\notin \mathcal{N},$ where $\mathcal{N}$ is defined in \eqref{critical value_cyl1}.  This condition is natural and expected, as a similar phenomenon arises already in the one-dimensional case (see Section~\ref{1d_discussion} for further discussion).
\item The control region $\Gamma$ is reduced to a vertical segment of positive length, as illustrated in \Cref{fig:region1}, thereby relaxing the two-edge assumption of earlier work.
\begin{figure}[htbp!]
	\centering
	\begin{tikzpicture}[scale=3] 
		\draw[thin, -] (-0.2,0) -- (0,0); 
		\draw[thin, ->] (1,0) -- (1.2,0) node[below] {$x$}; 
		
		\draw[thin, -] (0,-0.2) -- (0,0); 
		\draw[thin, ->] (0,1) -- (0,1.2) node[left] {$y$}; 
		
		\draw[ thick] (0,0) -- (1,0);
		\draw[thick] (1,0) -- (1,1);
		\draw[thick] (1,1) -- (0,1);
		\draw[ thick] (0,1) -- (0,0) ;
		
		\draw[ultra thick, red] (0,0.4) -- (0,0.75); 
		\node[left] at (0.2,0.575) {$\omega$}; 
		
		\node[below left] at (0,0) {0};
		\node[below] at (1,0) {a};
		\node[left] at (0,1) {b};
		\node at (0.6, 0.7) {\Large$\Omega$}; 
	\end{tikzpicture}
	\caption{Control region denoted by $\omega$ colored as red}
	\label{fig:region1}
\end{figure}
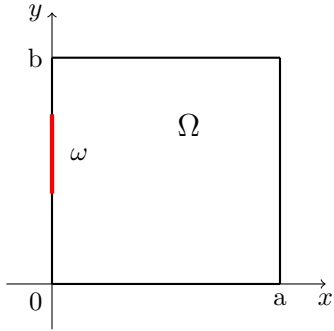
\item Moreover, for $N \geq 2 $, we establish necessary and sufficient conditions on the parameter $\nu$ for the null controllability of the linearized problem \eqref{lin_NKS_cyl}; see \Cref{Ln_KSE}. To the best of our knowledge, such necessary conditions had not been previously identified in the literature for the Kuramoto–Sivashinsky equation in higher dimensions.
\end{itemize}

\subsection{Strategy of the proofs of our main results}\label{st}
The starting point is to treat the linear setting with boundary controls. Our strategy reduces the higher-dimensional problem to a family of one-dimensional problems
by decomposing the solution of the associated adjoint system into eigenfunction expansions along one direction. The main idea is to combine the boundary control result in 1-D with the Lebeau-Robbiano strategy \cite{LR} to establish the null controllability of the main problem with an explicit control cost. This will give the proof of \Cref{Ln_KSE}. We remark that this strategy was developed for the first time in \cite{AB2014}. 

Next,  by applying the  source term method introduced in \cite{Tucsnak-nonlinear},  we prove a null-controllability result for the linearized model with additional source terms  that decay exponentially 
as $t\to T^-$, and in this step, we notably use the precise control cost obtained in \Cref{Ln_KSE}. Finally, we use the Banach fixed-point theorem to obtain the local (boundary)  null-controllability for our nonlinear system \eqref{th_nl_KS} and thereby completing the proof of \Cref{th_nl_KS}.

Let us discuss the strategy of the proof for the linear case in more detail.
\begin{itemize}[noitemsep, topsep=0.5mm, leftmargin=5.5pt]
	\item[--]We begin by employing the moment method to establish a boundary null controllability result for the one-dimensional Kuramoto-Sivashinsky equation (see \Cref{1d_discussion} below for a literature review on the associated 1-D control problems). For each $j\in\mathbb N$, let us consider
\begin{equation}\label{KS-oned_cyl}
	\begin{cases}
		\partial_t v+\partial_{x}^4 v+\left(\nu-2\mu^{\oy}_j\right)\partial_{x}^2 v=0 &   t\in (0,T), \;\; x\in \Omega_x,\\
		v(t,0)=v(t,a)=\partial_{x}^2 v(t,a)=0, \quad  \partial_{x}^2 v(t,0)=q(t) & t\in (0,T), \\
		v(0,x)=v_0(x) & x\in \Omega_x,
	\end{cases}
\end{equation}
where $\{\mu_j^{\Omega_y}\}_{j\in \N}$ denotes the set of eigenvalues of the Dirichlet Laplacian in the domain ${\Omega_y}.$ {For any fixed $j\in \mathbb N$, we define
\begin{equation}\label{critical value_oned}
	\mathcal{N}_j=\bigg\{2\mu^{\oy}_j+\pi^2\left(\frac{k^2+l^2}{a^2}\right): k,l \in \N, \;\; k\neq l\bigg\}.
\end{equation}}
The following theorem establishes the boundary null controllability of
\eqref{KS-oned_cyl}; {its proof is given in
\Cref{sec:teo16}.}

{
\begin{theorem}\label{null control 1d_cyl}
Let $T>0$ be given. Then, for every $j\in\mathbb N$ and each
$v_0\in L^2(\Omega_x)$, there exists a control $q\in L^2(0,T)$ such that
system \eqref{KS-oned_cyl} satisfies $v(T)=0$ if and only if
{$\nu\notin\mathcal{N}_j$.} Moreover, the control satisfies
\begin{equation}\label{cost_cyl}
	\norm{q}_{L^2(0,T)}
	\leq
	Ce^{\frac{C}{T^{1/3}}}
	\norm{v_0}_{L^2(\Omega_x)},
\end{equation}
where $C>0$ is independent of $T$, $j$, and $v_0$.
\end{theorem}
\begin{remark}
We point out that the control cost in~\eqref{cost_cyl} is uniform with respect to the parameter~$j$, even though the coefficient of the second-order term in equation~\eqref{KS-oned_cyl} depends on~$j$. This uniformity follows from combining the estimate of the biorthogonal family with the dissipative factor associated with the eigenvalues, as shown in the proof of \Cref{null control 1d_cyl}.
\end{remark}
}	
	
\begin{remark} We show that the KS equation in 1-D fails to be approximately controllable when the parameter $\nu \in \mathcal N_j$. 
Thus the condition {$\nu\notin \mathcal{N}_j$} is necessary for the null controllability, see \Cref{necessary} for details.
\end{remark}
\item[--] Next, we establish a partial observability result (see \Cref{par obs_cyl}) for the adjoint system corresponding to \eqref{lin_NKS_cyl}, with data in a subset of $L^2(\Omega)$. To construct the control in the multidimensional setting, we use the controllability result for the one-dimensional case (see \Cref{null control 1d_cyl}) together with the following spectral inequality:
\begin{theorem}\label{thm:spec_ineq_degen}
	Let $(\mu_j^{\Omega_y},\Psi_j^{\Omega_y})$ be the eigenpair  of the eigenvalue problem \eqref{eg_value}. Let $\omega$ be an open and nonempty subset of $\Omega_y$. There exists a constant $C > 0$ such that
	\begin{equation*}
		\sum_{\mu_j^{\Omega_y} \leq \mu} |a_j|^2 \leq C e^{C  \sqrt{\mu}} 
		\int_{\omega} \left| \sum_{\mu_j^{\Omega_y} \leq \mu} a_j \Psi_j^{\Omega_y} \right|^2 \, dy,
	\end{equation*}
	 $\{a_j\} \in \mathbb{R}$, and any $\mu > 0$.
\end{theorem}

\item[--] We utilize the partial observability mentioned in the previous step and combining it with the usual Lebeau-Robbiano approach, we design a control strategy driving $u$ to zero at time $T.$ This consists of the following: 
\begin{enumerate}
\item Deduce that we can find a control that kills the low frequencies.
\item Build a control that decreases the norm of $u(t)$, by first killing the low frequencies
with the control found in the previous step, and then let the system evolve with zero control and take advantage of the natural dissipation of the system.
\end{enumerate} 
\end{itemize}

\begin{remark}
It is worth mentioning that, despite the presence of the parameter~$j$ in the control cost of the 1-D problem, we are able to prove the controllability of the higher-dimensional system by exploiting the dissipation of the original system (see \Cref{dis_cyl}), which compensates for the growth in control cost. Notably, in the seminal paper \cite{AB2014} (see also \cite{BO24,khodja2024new,HSMdT25}), there is no dependency on the parameter~$j$ in the control cost coming from the 1-D problem. This phenomenon appears to be purely related to the fourth-order operator.
\end{remark}
We apply the same technique discussed above to prove \Cref{thm_int}. In \cite{S15}, the author studied a similar result as \Cref{thm_int} for the $N$-dimensional heat equation. Using a pointwise controllability result for one-dimensional heat equation \cite{D73} and then applying Lebeau-Robbiano technique, {he concludes} the result. On the contrary, to prove our result \eqref{thm_int}, we need to demonstrate first an analogous pointwise controllability result for the KS equation (see \Cref{int_nullcontrol} and \Cref{int_thm_on}). Then following the technique of \cite{S15}, we derive the main result.

\subsection{Bibliographic details for 1-D KS equation}\label{1d_discussion}
For the sake of bibliographical completeness, in this section we discuss a brief overview of the existing literature concerning one-dimensional KS equation. The earliest study regarding controllability 
of KS equation in 1-D appears in
\cite{EC10}, where the author employed the moment method to investigate the 
boundary null controllability of linear KS 
equation
\begin{equation*}
	\begin{cases}
		\partial_t u+\partial_{x}^4 u+\nu\partial_{x}^2 u=0 & \quad t\in (0,T), \; x\in (0,1),\\
		u(t,0)=u(t,1)=\partial_{x} u(t,1)=0, \partial_{x} u(t,0)=q(t)& \quad t\in (0,T), \\
		u(0,x)=u_0(x) & \quad x\in (0,1),
	\end{cases}
\end{equation*}
with a single control force $q_1\in H^1(0,T)$ acting 
on first order derivative at left end point for any $u_0\in L^2(0,1)$ 
 provided $\nu\notin \mathcal{N}_1$, where
\begin{equation*}
	\mathcal{N}_1=\bigg\{\pi^2\left({k^2+l^2}\right): k,l \in \N, \;\; k\neq l, k\equiv l \text{ mod } 2\bigg\}.
\end{equation*}
The authors in
\cite{EC11} considered  nonlinear 
KS equation in $(0,T)\times (0,1)$
\begin{equation}\label{KS-oned2}
	\begin{cases}
		\partial_t u+\partial_{x}^4 u+\nu\partial_{x}^2 u+uu_x=0,  & \\
		u(t,0)=q_1(t),\,  \partial_{x} u(t,0)=q_2(t), \, u(t,1)=\partial_{x} u(t,1)=0,\\
		u(0,x)=u_0(x), 
	\end{cases}
\end{equation}
and proved its local exact controllability 
to trajectory with $L^2(0,T)$ boundary controls 
acting at left endpoint of zeroth and first 
order derivative without any critical set conditions for $\nu$. 
They first used Carleman estimates to study the null controllability of
the linearized KS equation and then applied a local inversion theorem
to obtain the desired result for the nonlinear KS equation.

The study of the null controllability of the linear KS equation was
further developed in \cite{EC17}, where the authors investigated its
boundary null controllability with a control acting only on the
zeroth-order derivative at the left endpoint, using the moment method.
Next, they considered the Neumann boundary case and proved that the
linear KS system is not null controllable with a single control acting
on either the second- or third-order derivative, but becomes null
controllable when both boundary controls act simultaneously on the
system.

The author of \cite{TT17} studied the local boundary null
controllability of the KS equation \eqref{KS-oned2} with
$q_1=0$, $q_2\in L^2(0,T)$, and $u_0\in H^{-1}(0,1)$, provided that
$\nu\notin\mathcal{N}_1$. As in the two-dimensional case discussed in
\Cref{sec_tak}, Russell's method (see \cite{R73}) and the source-term
method (see \cite{Tucsnak-nonlinear}) were used.

Let us conclude our literature review by mentioning \cite{CL15}, where
the authors studied the Fredholm transformation and local rapid
stabilization for the KS equation
\begin{equation}\label{KS-oned3}
\begin{cases}
\partial_t u+\partial_x^4u+\nu\partial_x^2u+u\partial_xu=0
& \text{in }(0,T)\times(0,1),\\
u(t,0)=u(t,1)=\partial_x^2u(t,1)=0,\quad
\partial_x^2u(t,0)=q(t)
& t\in(0,T),\\
u(0,x)=u_0(x)
& x\in(0,1),
\end{cases}
\end{equation}
provided that $\nu\notin\mathcal{N}_2$, where
\begin{equation}\label{critical value3}
\mathcal{N}_2
=
\bigg\{
\pi^2\left(k^2+l^2\right):
k,l\in\N,\;\;k\neq l
\bigg\}.
\end{equation}
Note that the one-dimensional system \eqref{KS-oned_cyl} associated
with our problem is closely related to \eqref{KS-oned3} in terms of
its boundary conditions. Thus, the critical set condition
\eqref{critical value_cyl1} can be naturally anticipated from
\eqref{critical value3}.

\subsection{Outline of the paper} The rest of the paper is organized as follows. \Cref{sec:framework} is devoted to the functional framework and well-posedness of the higher-dimensional linear KS system. In \Cref{1d}, we establish that the one dimensional problem is null controllable with an explicit cost of the control (see \Cref{null control 1d_cyl}). \Cref{sec:cylinder} contains the proof of the main result in the {linear} setting (\Cref{Ln_KSE}). \Cref{sec:internal} deals with the internal control problem in higher dimensions. Finally, in \Cref{sec:nonlinear}, we present a brief proof of the nonlinear result stated in \Cref{th_nl_KS}.

\section{Functional framework and well-posedness}\label{sec:framework}
The goal of this section is to establish the well-posedness of the linearized system \eqref{lin_NKS_cyl}. This will be a consequence of a general result for the following equation
\begin{equation}\label{lin_NKS_gen}
	\begin{cases}
		\partial_t u+\Delta^2 u+\nu\Delta u=f &  \text{ in } (0,T)\times \Omega,\\
		u=0, \quad \Delta u=\mathbf{1}_{\Gamma}q & \text{ on } (0,T)\times \partial \Omega, \\
		u(0)=u_0 & \text{ in } \Omega,
	\end{cases}
\end{equation}
where $u_0$, $f$, $q$ are functions taken from suitable functional spaces. We obtain this in several steps. 

Let $\nu>0$ be given and define the operator $\mathcal{A}:\mathcal D(\mathcal{A})\subset L^2(\Omega)\to L^2(\Omega)$ as
\begin{equation}\label{op}
\begin{cases}
\mathcal D(\mathcal A)=\{u\in H^4(\Omega):u=\Delta u=0 \;\; \textnormal{on}\;\;\partial \Omega\} \\
\mathcal Au=-\Delta^2u-\nu \Delta u
\end{cases}
\end{equation}
We begin by proving the following lemma.
\begin{lemma}\label{lem:opA}
The operator $\mathcal A$ is self-adjoint, quasi-dissipative, and maximal.
\end{lemma}
\begin{proof}
The first property is clear. To prove that $\mathcal A$ is quasi-dissipative, let us compute by integration by parts and Young inequality
\begin{align*}
(\mathcal A u, u)_{L^2(\Omega)}&=-\int_{\Omega}u\Delta^2u-\nu\int_{\Omega}u \Delta u = -\int_{\Omega}|\Delta u|^2-\nu\int_{\Omega}u \Delta u \\
&\leq -\frac{1}{2}\int_{\Omega}|\Delta u|^2+\frac{\nu^2}{2}\int_{\Omega}|u|^2 \leq \frac{\nu^2}{2}\int_{\Omega}|u|^2,
\end{align*}
which proves our claim. 

For the maximality, we show that for some $\lambda$ large enough, $R(\lambda I-\mathcal A)=L^2(\Omega)$. To prove it, we show that for any $f\in L^2(\Omega)$ and any $\lambda>\frac{\nu^2}{2}$, there is $u\in \mathcal D(\mathcal A)$ satisfying the elliptic boundary-value problem
\begin{equation*}
\lambda u+ \Delta^2 u+\nu \Delta u=f \text{ in } \Omega, \quad u=\Delta u=0 \text{ in } \partial \Omega.
\end{equation*}
Let $a:H^2(\Omega)\cap H_0^1(\Omega)\times H^2(\Omega)\cap H_0^1(\Omega)\to \mathbb R$ be the bilinear form given by
\begin{equation*}
a(u,v)=\lambda u+ \int_{\Omega}\Delta u \Delta v+\nu \int_{\Omega}\Delta u v, \qquad u,v\in H^2(\Omega)\cap H_0^1(\Omega).
\end{equation*}
Clearly, $a$ is continuous and, moreover, for any $u\in H^2(\Omega)\cap H_0^1(\Omega)$, we can argue as above and see that
\begin{equation*}
a(u,u)\geq \frac{1}{2}\int_{\Omega}|\Delta u|^2+\left(\lambda-\frac{\nu^2}{2}\right)\int_{\Omega}|u|^2 \geq c_0\|u\|^2_{H^2(\Omega)\cap H_0^1(\Omega)},
\end{equation*}
for some positive constant $c_0$. Therefore, by Lax-Milgram Lemma, we have that for any $f\in L^2(\Omega)$, there is a unique $u\in H^2(\Omega)\cap H_0^1(\Omega)$ such that $a(u,v)=\int_{\Omega}fv$ for all $v\in H^2(\Omega)\cap H_0^1(\Omega)$. By standard arguments, it can be derived that $u\in \mathcal D(\mathcal A)$, and thus $R(\lambda I-\mathcal A)=L^2(\Omega)$. This ends the proof.
\end{proof}

In turn, this yields the well-posedness of the following uncontrolled equation 
\begin{equation}\label{lin_NKS_cyl_homo}
	\begin{cases}
		\partial_t u+\Delta^2 u+\nu\Delta u=f &  \text{ in } (0,T)\times \Omega,\\
		u=0, \quad \Delta u=0 & \text{ on } (0,T)\times \partial \Omega, \\
		u(0)=u_0 & \text{ in } \Omega.
	\end{cases}
\end{equation}
To ease the reading, from now on, we denote by $\mathcal H(\Omega):= H^2(\Omega)\cap H_0^1(\Omega)$ and $\mathcal H^\prime(\Omega)$ its dual (with respect to the pivot space $L^2(\Omega)$). The result reads as follows.
\begin{proposition}\label{prop:wp_noncontr}
For any $u_0\in L^2(\Omega)$ and $f\in L^2(0,T;\mathcal H'(\Omega))$, there exists a unique solution to \eqref{lin_NKS_gen} (with $q\equiv 0$) such that $u\in C([0,T];L^2(\Omega))\cap L^2(0,T;\mathcal H(\Omega))$. Moreover, there exists $C>0$ independent of $f$, $T$, and $u_0$, such that
\begin{equation}\label{eq:apriori}
\begin{split}
&\|u\|_{C([0,T];L^2(\Omega))}+\|u\|_{L^2(0,T;\mathcal H(\Omega))}+\|\pa_t u\|_{L^2(0,T;\mathcal H^\prime(\Omega))}\\
&\quad \leq Ce^{CT}\left(\|u_0\|_{L^2(\Omega)}+\|f\|_{L^2(0,T;\mathcal H^\prime(\Omega))}\right).
\end{split}
\end{equation}
\end{proposition}
\begin{proof}
By \Cref{lem:opA} and \cite[Theorem 12.22]{RR04}, $(\mathcal A,D(\mathcal A))$ generates a strongly continuous semigroup on $L^2(\Omega)$. Consequently, by Theorem 11.3 and Lemma 11.4 from \cite{RR04}, system \eqref{lin_NKS_cyl_homo} has a unique solution $u\in C([0,T];L^2(\Omega))\cap L^2(0,T;\mathcal H(\Omega))\cap H^1(0,T;\mathcal H^\prime(\Omega))$ satisfying the a priori estimate \eqref{eq:apriori}. 
\end{proof}

As a next step, we introduce the adjoint equation to \eqref{lin_NKS_cyl}, more precisely
\begin{equation}\label{lin_NKS_adj}
	\begin{cases}
		-\partial_t \sigma +\Delta^2 \sigma +\nu\Delta \sigma=0 &  \text{ in } (0,T)\times \Omega,\\
		\sigma=0, \quad \Delta \sigma=0 & \text{ on } (0,T)\times \partial \Omega, \\
		\sigma(T)=\sigma_T & \text{ in } \Omega,
	\end{cases}
\end{equation}
where $\sigma_T\in L^2(\Omega)$ is a terminal datum. The following result holds. 
\begin{proposition}\label{lem:reg_backward}
For any $\sigma_T\in L^2(\Omega)$, there exists a unique solution to \eqref{lin_NKS_adj} such that $\sigma\in C([0,T];L^2(\Omega))\cap L^2(0,T;\mathcal H(\Omega))$ and, in addition, there is $C>0$ independent of $T$ and $\sigma_T$ such that
\begin{equation*}
\|\sigma\|_{C([0,T];L^2(\Omega))}+\|\sigma\|_{L^2(0,T;\mathcal H(\Omega))}+\|\pa_t\sigma\|_{L^2(0,T;\mathcal H^\prime(\Omega))}\leq Ce^{CT}\|\sigma_T\|_{L^2(\Omega)}.
\end{equation*}
\end{proposition}
The proof of this result can be deduced from \Cref{prop:wp_noncontr} and a change of variable in time.  This motivates the following notion of solution to \eqref{lin_NKS_gen}.

\begin{definition}
Let $u_0\in L^2(\Omega)$, $f\in L^2(0,T;\mathcal H^\prime(\Omega))$, and $q\in L^2(0,T;L^2(\Gamma))$ be given. A function $u\in C([0,T];L^2(\Omega))$ is said to be a solution by transposition to \eqref{lin_NKS_gen} if for each $\sigma_\tau\in L^2(\Omega)$ and for all $\tau\in(0,T]$ one has
\begin{equation}\label{iden_transposition}
\begin{split}
\langle u(\tau), \sigma_\tau \rangle_{L^2(\Omega)}&=\langle u_0, \sigma(0,\cdot) \rangle_{L^2(\Omega)}+\int_0^{\tau}\langle f(s),\sigma(s)\rangle_{\mathcal H^\prime(\Omega),\mathcal H(\Omega)} ds \\
&\quad +{\int_{0}^{\tau}}\int_{ \Gamma}q(s)\frac{\partial \sigma}{\partial \eta}(s) \, dS ds,
\end{split}
\end{equation}
where $\eta$ is the unitary outward normal vector, $dS$ is the surface measure, and $\sigma$ is a solution to \eqref{lin_NKS_adj} with $\sigma(\tau,\cdot)=\sigma_{\tau}(\cdot)$ in $\Omega$. In \eqref{iden_transposition}, $\langle \cdot,\cdot \rangle_{\mathcal H^\prime(\Omega),\mathcal H(\Omega)}$ stands for the duality pairing between $\mathcal H(\Omega)$ and $\mathcal H^\prime(\Omega)$.
\end{definition}

We are in a position to state our main well-posedness result. 

\begin{proposition}\label{prop:wellposedness}
	Let $u_0 \in \Ltwo$, $f \in L^2(0,T;\Hdual)$, and
	$q \in L^2(0,T;\LtwoGamma)$. Then equation \eqref{lin_NKS_cyl}
	admits a unique solution
	$u \in C([0,T];\Ltwo) \cap L^2(0,T;\Hsp)$, defined in the sense
	of transposition. Moreover, $u$ satisfies the energy estimate
	\begin{align}\label{eq:apriori_full}
		\notag\|u\|_{C([0,T];\Ltwo)} &+ \|u\|_{L^2(0,T;\Hsp)}+\|\pa_t u\|_{L^2(0,T;\Hdual)}\\ &\leq C e^{CT} \left( \|u_0\|_{\Ltwo} + \|f\|_{L^2(0,T;\Hdual)} + \|q\|_{L^2(0,T;\LtwoGamma)} \right).
	\end{align}
\end{proposition}
\begin{proof}
{	For ease of reading, we split the proof into several steps.
	We reformulate system \eqref{lin_NKS_cyl} into an abstract Cauchy problem

		\noindent\textbf{Step 1: Abstract operator setup.}
		Recall the linear operator $\mathcal{A} \colon \mathcal{D}(\mathcal{A}) \subset L^2(\Omega) \to L^2(\Omega)$ defined in \eqref{op}.
	It is straightforward to verify that
	there exists a constant $\lambda_0 \ge 0$ such that the shifted operator $\mathcal{A}_0 := -\mathcal{A} + \lambda_0 I$ is self-adjoint, strictly positive, and possesses a compact resolvent on $L^2(\Omega)$. Consequently,  $\mathcal{A}$ generates an analytic semigroup $S(t) = e^{\mathcal{A}t}$ on $L^2(\Omega)$.
		
		The domain of the fractional power $\mathcal{A}_0^{1/2}$ is given by
	$ \mathcal{D}\left(\mathcal{A}_0^{1/2}\right) = H^2(\Omega) \cap H_0^1(\Omega)=\mathcal{H}(\Omega),$
		which we equip with the inner product $\langle u, v \rangle_{\mathcal{H}(\Omega)} = \int_{\Omega} \Delta u \, \Delta v \, dx$ and associated norm $\|v\|_{\mathcal{H}(\Omega)} = \|\Delta v\|_{L^2(\Omega)}$ (equivalent to the standard $H^2(\Omega)$-norm under Navier boundary conditions), see \cite[Section 3.4]{TW09} for more details. 
		By standard semigroup extrapolation theory \cite[Corollary 3.4.6]{TW09}, $\mathcal{A}$ extends uniquely to a bounded linear operator $\mathcal{A} \colon \mathcal{H}(\Omega) \to \mathcal{H}'(\Omega)$. 
		Equivalently, $\mathcal{A}$ induces a continuous bilinear form $a \colon \mathcal{H}(\Omega) \times \mathcal{H}(\Omega) \to \mathbb{R}$ defined by
		\begin{equation*}
			a(u,v) := -\langle \mathcal{A}u, v \rangle_{\mathcal{H}'(\Omega), \mathcal{H}(\Omega)} = \int_{\Omega} \Delta u \Delta v \, dx + \nu \int_{\Omega} \Delta u  v \, dx, \qquad u, v \in \mathcal{H}(\Omega).
		\end{equation*}
		By applying Young's inequality to the lower-order term, $a(\cdot,\cdot)$ satisfies the coercivity condition:
		\begin{equation*}
			a(u,u) \ge c_1 \|u\|^2_{\mathcal{H}(\Omega)} - c_2 \|u\|^2_{L^2(\Omega)} \quad \forall \, u \in \mathcal{H}(\Omega),
		\end{equation*}
		for suitable positive constants $c_1, c_2 > 0$.
		
		\noindent\textbf{Step 2: Representation of boundary control.}
Let us introduce the boundary control operator $\mathcal B \colon L^2(\Gamma) \to \mathcal{H}'(\Omega)$ given by
		\begin{equation}\label{def_B_operator}
			\langle \mathcal B q, v \rangle_{\mathcal{H}'(\Omega), \mathcal{H}(\Omega)} = \int_{\Gamma} q \frac{\partial v}{\partial \eta} \, dS \quad \forall \, v \in \mathcal{H}(\Omega).
		\end{equation}
		By the Sobolev trace theorem, the normal derivative map $\frac{\partial}{\partial \eta} \colon \mathcal H(\Omega) \to L^2(\partial \Omega)$ is continuous with $\big\|\frac{\partial v}{\partial \eta}\big\|_{L^2(\partial \Omega)} \leq C \|v\|_{\mathcal{H}(\Omega)}$. Thus, $\mathcal B$ is a well-defined bounded operator:
		\begin{equation*}
			\|\mathcal B q\|_{\mathcal{H}'(\Omega)} \leq C \|q\|_{L^2(\Gamma)} \implies \mathcal B \in \mathcal{L}\left(L^2(\Gamma), \mathcal{H}'(\Omega)\right).
		\end{equation*}
Let us recall that system \eqref{lin_NKS_cyl} can be written equivalently in the following abstract form in $\mathcal{H}'(\Omega)$:
\begin{equation}\label{abstract_ode}
	\begin{cases}
		\displaystyle\frac{du}{dt}(t)-\mathcal{A}u(t)
		=f(t)+\mathcal{B}q(t),
		& t\in(0,T),\\
		u(0)=u_0\in L^2(\Omega).
	\end{cases}
\end{equation}
	Since $f\in L^2(0,T;\Hdual)$ and
$\mathcal Bq\in L^2(0,T;\Hdual)$, we have
$f+\mathcal Bq\in L^2(0,T;\Hdual)$. Thus, by
\cite[Theorem 11.3]{RR04}, system \eqref{abstract_ode} has a unique
solution $u\in L^2(0,T;\Hsp)\cap H^1(0,T;\Hdual)$. The required time
continuity $u\in C([0,T];\Ltwo)$ follows from the standard embedding. To derive the a priori estimate \eqref{eq:apriori_full}, we test \eqref{abstract_ode} against
$u$, integrate over $(0,T)$, and apply Gronwall's inequality. This
yields
\begin{equation}\label{eq:apriori_full1}
\begin{split}
&\norm{u}_{C([0,T];\Ltwo)}
+\norm{u}_{L^2(0,T;\Hsp)}
\\
&\quad\leq
Ce^{CT}
\left(
\norm{u_0}_{\Ltwo}
+\norm{f}_{L^2(0,T;\Hdual)}
+\norm{q}_{L^2(0,T;\LtwoGamma)}
\right).
\end{split}
\end{equation}
Similarly, testing \eqref{abstract_ode} against $v\in\Hsp$ gives
\begin{align*}
\left|\langle\pa_tu,v\rangle_{\Hdual,\Hsp}\right|
\leq C\left(
\norm{u}_{\Hsp}
+\norm{f}_{\Hdual}
+\norm{\mathcal Bq}_{\Hdual}
\right)\norm{v}_{\Hsp},
\qquad v\in\Hsp.
\end{align*}
Using the continuity of $\mathcal B:\LtwoGamma\to\Hdual$ and
\eqref{eq:apriori_full1}, we obtain \eqref{eq:apriori_full}.

	\noindent\textbf{Step 3: Equivalence with the transposition formulation.}
	It remains to verify that the solution $u$ satisfies the transposition identity \eqref{iden_transposition}. Fix $\tau \in (0,T]$, and let $\sigma$ be the unique solution to the adjoint system \eqref{lin_NKS_adj} with terminal data $\sigma(\tau) = \sigma_\tau \in \Ltwo$. In its abstract formulation, $\sigma$ satisfies $-\frac{d\sigma}{dt}(t) =\mathcal{A} \sigma(t)$ in $\Hdual$. Taking the duality pairing of \eqref{abstract_ode} with $\sigma(t) \in \Hsp$ and integrating over $(0,\tau)$ yields
	\begin{align*}
		\int_0^\tau \left\langle \frac{du}{dt}(t) - \mathcal{A}u(t), \sigma(t) \right\rangle_{\Hdual, \Hsp} dt = \int_0^\tau \langle f(t) + \mathcal Bq(t), \sigma(t) \rangle_{\Hdual, \Hsp} dt.
	\end{align*}
	Integrating the time derivative by parts and using the self-adjointness of $\mathcal{A}$, the left-hand side reduces to
	\begin{align*}
		\langle u(\tau), \sigma_\tau \rangle_{\Ltwo} - \langle u_0, \sigma(0) \rangle_{\Ltwo} - \int_0^\tau \left\langle u(t), \frac{d\sigma}{dt}(t) + \mathcal{A}\sigma(t) \right\rangle_{\Hdual, \Hsp} dt.
	\end{align*}
	Because $\sigma$ solves the adjoint equation $\frac{d\sigma}{dt} + \mathcal{A}\sigma = 0$, 
	 using the definition of $\mathcal Bq$ from \eqref{def_B_operator}  yields
	\begin{equation*}
	\begin{split}
		\langle u(\tau), \sigma_\tau \rangle_{\Ltwo} &= \langle u_0, \sigma(0) \rangle_{\Ltwo} + \int_0^\tau \langle f(s), \sigma(s) \rangle_{\Hdual,\Hsp} ds \\
		&\quad + \int_0^\tau \int_\Gamma q(s) \frac{\partial \sigma}{\partial \eta}(s) \, dS \, ds,
		\end{split}
	\end{equation*}
	which is precisely \eqref{iden_transposition}.}
\end{proof}
\section{Boundary controllability problem of the one-dimensional KS equation}\label{1d}
This section is devoted to the study of the null controllability of the 1-D KS equation \eqref{KS-oned_cyl}, i.e., \Cref{null control 1d_cyl}. We begin by stating an existence result for this equation, whose proof follows the same approach as that presented in \cref{sec:framework}.
\subsection{Linearized operator and well-posedness}
Let us first define the unbounded linear operator $A : \mathcal D(A)\subset L^2(\Omega_x)\mapsto L^2(\Omega_x)$ as follows:
\begin{equation*}
	\begin{cases}
		\mathcal{D}( A) = \{u \in H^4(\Omega_x) : u(0) = u(a) = \pa_x^2  u(0) = \pa_x^2 u(a) = 0 \},\\
		Au =-\partial_{x}^4 u-\left(\nu-2\mu^{\oy}_j\right)\partial_{x}^2 u , \quad \forall \, u \in \mathcal D(A). \end{cases}
\end{equation*}
$A$ is a self-adjoint operator with compact resolvent.
Simple computations give that the
eigenvalues of $A$ are
\begin{equation}\label{eigenv}
	\lambda^{\ox}_k = -\frac{k^4\pi^4}{a^4}+\left(\nu-2\mu^{\oy}_j\right) \frac{k^2\pi^2}{a^2}, \quad   \forall \, k \in \N
\end{equation}
and the eigenfunction corresponding to the eigenvalue $\lambda^{\ox}_k$
is
\begin{equation*}\Psi^{\ox}_k (x) = \sqrt{\frac{2}{a}}
\sin\left(\frac{k\pi x}{a} \right), \quad k\in \N.\end{equation*}
Next, we write the adjoint system of \eqref{KS-oned_cyl} as follows:
\begin{equation}\label{adj_cyl}
	\begin{cases}
		-\partial_t \phi+\partial_{x}^4 \phi+\left(\nu-2\mu^{\oy}_j\right)\partial_{x}^2 \phi=0 & \quad  t\in (0,T), \;\; x\in \Omega_x,\\
		\phi(t,0)=\phi(t,a)=0,\quad  \partial_{x}^2 \phi(t,0)=\partial_{x}^2 \phi(t,a)=0 & \quad t\in (0,T), \\
		\phi(T,x)=\phi_T(x) & \quad x\in \Omega_x.
	\end{cases}
\end{equation}
\begin{proposition}\label{lem:reg_backward1}
	For any $\phi_T\in L^2(\Omega_x)$, there exists a unique solution to \eqref{adj_cyl} such that $\phi\in C([0,T];L^2(\Omega_x))\cap L^2(0,T;\mathcal H(\Omega_x))$ and, in addition, there is $C>0$ independent of $T$ and $\phi_T$ such that
	\begin{equation*}
		\|\phi\|_{C([0,T];L^2(\Omega))}+\|\phi\|_{L^2(0,T;\mathcal H(\Omega_x))}+\|\pa_t\phi\|_{L^2(0,T;\mathcal H^\prime(\Omega_x))}\leq Ce^{CT}\|\phi_T\|_{L^2(\Omega_x)}.
	\end{equation*}
\end{proposition}
Let us define the following notion of solution to \eqref{KS-oned_cyl}.
\begin{definition}[Solution by transposition] Let $v_0\in L^2(\Omega_x)$, and $q\in L^2(0,T)$ be given. A function $v\in C([0,T];L^2(\Omega_x))$ is said to be a solution by transposition to \eqref{KS-oned_cyl} if for each $\phi_\tau\in L^2(\Omega_x)$ and for all $\tau\in(0,T]$ one has
	\begin{equation*}
		\langle v(\tau), \phi_\tau \rangle_{L^2(\Omega_x)}=\langle v_0, \phi(0,\cdot) \rangle_{L^2(\Omega_x)} -\int_{0}^{\tau} q(s)\phi_x(s,0) ds.
	\end{equation*}
\end{definition}
We are in a position to state the well-posedness result for \eqref{KS-oned_cyl}. 
\begin{proposition}
	For every $v_0\in L^2(\Omega_x)$ and $q\in L^2(0,T)$, equation  \eqref{KS-oned_cyl} has a unique solution $v\in C([0,T];L^2(\Omega_x))$ defined in the sense of transposition and satisfies the following estimate
	\begin{align*} 
		\|v\|_{C([0,T];L^2(\Omega_x))}&+\|v\|_{L^2(0,T;\mathcal H(\Omega_x))}+\|\pa_t v\|_{L^2(0,T;\mathcal H^\prime(\Omega_x))}\\ 
		&\quad \leq Ce^{CT}\left(\|v_0\|_{L^2(\Omega_x)}+\|q\|_{L^2(0,T)}\right).
	\end{align*}
\end{proposition}
\subsection{Approximate controllability}
We start our study concerning the control problem of \eqref{KS-oned_cyl} by proving the following approximate controllability result:
\begin{theorem}\label{aprx}
System \eqref{KS-oned_cyl} is approximately controllable if and only if {$\nu\notin \mathcal{N}_j.$}
\end{theorem}
\begin{proof}
We recall that the approximate controllability of system \eqref{KS-oned_cyl} 
is equivalent to showing that the solution of  
\begin{equation}\label{KS-oned aprx_cyl}
	\begin{cases}
		\partial_t v+\partial_{x}^4 v+\left(\nu-2\mu^{\oy}_j\right)\partial_{x}^2 v=0 & \quad  t\in (0,T), \;\;x\in \Omega_x,\\
		v(t,0)=v(t,a)=0, \quad \partial_{x}^2 v(t,0)=\partial_{x}^2 v(t,a)=0 & \quad t\in (0,T), \\
		v_x(t,0)=0 & \quad  t\in (0,T),\\
		v(0,x)=v_0(x) & \quad x\in \Omega_x,
	\end{cases}
\end{equation}
is identically zero for any initial condition $v_0 \in L^2(\Omega_x)$. 
Note that in the above system, the zero observation condition has been incorporated 
by imposing $v_x(t,0)=0$.

	{ Let us first fix any $j\in \N$ and assume that $\nu\notin \mathcal{N}_j$.} The solution of equation \eqref{KS-oned aprx_cyl} can be written as
	\begin{equation*}
		v(t,x)=\sum_{k=1}^{\infty}v_{0,k}e^{\left(-\frac{k^4\pi^4}{a^4}+\left(\nu-2\mu^{\oy}_j\right) \frac{k^2\pi^2}{a^2}\right)t}\sqrt{\frac{2}{a}}\sin\left(\frac{k\pi x}{a}\right),
	\end{equation*}
	where $v_{0,k}=\sqrt{\frac{2}{a}}\int_{0}^{a} v_0(x)\sin\frac{k\pi x}{a}.$
	In turn, the observation term takes the form
	\begin{equation*}
		v_x(t,0)=\sum_{k=1}^{\infty}v_{0,k}e^{\left(-\frac{k^4\pi^4}{a^4}+\left(\nu-2\mu^{\oy}_j\right) \frac{k^2\pi^2}{a^2}\right)t}\sqrt{\frac{2}{a}}\left(\frac{k\pi}{a}\right)=0.
	\end{equation*}
	and, since $v_x(t,0)$ is analytic in $(0,T]$, we have the following extension
	\begin{equation}\label{vx_cyl}
		v_x(t,0)=\sum_{k=1}^{\infty}v_{0,k}e^{\left(-\frac{k^4\pi^4}{a^4}+\left(\nu-2\mu^{\oy}_j\right) \frac{k^2\pi^2}{a^2}\right)t}\sqrt{\frac{2}{a}}\left(\frac{k\pi}{a}\right)=0 \quad \forall t\in (0,\infty).
	\end{equation}
	Let $k_0\in \N$ be such that 
	$-\frac{k_0^4\pi^4}{a^4}+\left(\nu-2\mu^{\oy}_j\right) \frac{k_0^2\pi^2}{a^2}=\max\limits_{k\in \N}\bigg\{-\frac{k^4\pi^4}{a^4}+\left(\nu-2\mu^{\oy}_j\right) \frac{k^2\pi^2}{a^2} \bigg\}.$ As we have assumed  {$\nu\notin \mathcal{N}_j$}, such $k_0$ is unique. Multiplying both sides of \eqref{vx_cyl} by $e^{\left(\frac{k_0^4\pi^4}{a^4}-\left(\nu-2\mu^{\oy}_j\right) \frac{k_0^2\pi^2}{a^2}\right)t}$ and letting $t\to \infty$, we obtain $v_{0,k_0}=0.$ Next, let $k_1\in \N$ be such that $-\frac{k_1^4\pi^4}{a^4}+\left(\nu-2\mu^{\oy}_j\right) \frac{k_1^2\pi^2}{a^2}=\max\limits_{k\in \N-\{k_0\}}\bigg\{-\frac{k^4\pi^4}{a^4}+\left(\nu-2\mu^{\oy}_j\right) \frac{k^2\pi^2}{a^2} \bigg\}$. Again, uniqueness of $k_1$ follows from the assumption {$\nu\notin \mathcal{N}_j$}. Multiplying both sides of \eqref{vx_cyl} by $e^{\left(\frac{k_1^4\pi^4}{a^4}-\left(\nu-2\mu^{\oy}_j\right) \frac{k_1^2\pi^2}{a^2}\right)t}$, we deduce that $v_{0,k_1}=0.$ Thus, by induction, we can prove that $v_{0,k}=0$, $\forall k\in \N.$ Therefore $v=0 \text{ in } (0,T)\times \Omega_x.$
	
	Conversely let {$\nu \in \mathcal{N}_j.$} Then there exists some $k_0\neq l_0$ such that $\nu=2\mu^{\oy}_j+\pi^2\left(\frac{k_0^2+l_0^2}{a^2}\right).$ Therefore, we have
	\begin{align*}
		-\frac{k_0^4\pi^4}{a^4}+\left(\nu-2\mu^{\oy}_j\right)\frac{k_0^2\pi^2}{a^2}=-\frac{l_0^4\pi^4}{a^4}+\left(\nu-2\mu^{\oy}_j\right) \frac{l_0^2\pi^2}{a^2}=\frac{k_0^2l_0^2\pi^4}{a^4}.
	\end{align*}
	Next, observe that 
	\begin{equation*}
	\begin{split}
	v(t,x)&=e^{\left(-\frac{k_0^4\pi^4}{a^4}+\left(\nu-2\mu^{\oy}_j\right) \frac{k_0^2\pi^2}{a^2}\right)t}\sin\left(\frac{k_0\pi x}{a}\right)\\
	&\quad -e^{\left(-\frac{l_0^4\pi^4}{a^4}+\left(\nu-2\mu^{\oy}_j\right) \frac{l_0^2\pi^2}{a^2}\right)t}\frac{k_0}{l_0}\sin\left(\frac{l_0\pi x}{a}\right)
	\end{split}
	\end{equation*}
	satisfies equation \eqref{KS-oned aprx_cyl} without $v=0$. {{This ends the proof}.}
\end{proof}	

In the next section, we will prove that the condition {$\nu\notin \mathcal{N}_j$} is sufficient for the null controllability of the system \eqref{KS-oned_cyl} at time $T>0$. Nevertheless, one can show that the condition {$\nu\notin \mathcal{N}_j$} is necessary for null controllability of \eqref{KS-oned_cyl}.
	\begin{corollary}\label{necessary}
		Let us assume that the system \eqref{KS-oned_cyl} is null controllable in time $T>0$. Then {$\nu\notin \mathcal{N}_j$}.
	\end{corollary}
\begin{proof}
	Let us assume that the system \eqref{KS-oned_cyl} is null controllable in time $T>0.$ Therefore for any $v_0\in L^2(\ox)$, there exists a control $q\in L^2(0,T)$ such that the following system 
	\begin{equation*}
			\begin{cases}
			\partial_t w+\partial_{x}^4 w+\left(\nu-2\mu^{\oy}_j\right)\partial_{x}^2 w=0 & \quad  t\in (0,T), \;\; x\in \Omega_x,\\
			w(t,0)=w(t,a)=0, \quad \partial_{x}^2 w(t,0)=q(t), \;\; \partial_{x}^2 w(t,a)=0& \quad t\in (0,T), \\
			w(0,x)=v_0(x)-\sqrt{\frac{2}{a}}\sin(\frac{k\pi x}{a}) & \quad x\in \Omega_x,
		\end{cases}
	\end{equation*}
satisfies $w(T,x)=0$, $\forall k\in \N.$
Next, it is easy to observe that the solution of the following system 
\begin{equation*}
	\begin{cases}
		\partial_t y+\partial_{x}^4 y+\left(\nu-2\mu^{\oy}_j\right)\partial_{x}^2 y=0 & \quad  t\in (0,T), \;\; x\in \Omega_x,\\
		y(t,0)=y(t,a)=0, \quad \partial_{x}^2 y(t,0)=\partial_{x}^2 y(t,a)=0 & \quad t\in (0,T), \\
		y(0,x)=\sqrt{\frac{2}{a}}\sin(\frac{k\pi x}{a}) & \quad x\in \Omega_x,
	\end{cases}
\end{equation*}
satisfies $y(T,x)=\sqrt{\frac{2}{a}}e^{\lambda_k^{\ox} T}\sin(\frac{k\pi x}{a}),$ as $\sin(\frac{k\pi x}{a})$  is an eigenfunction of the operator $A$ with respect to the eigenvalue $\lambda_k^{\ox}.$ Thus, it follows that $v=y+w$ solves equation \eqref{KS-oned_cyl} with $v(T,x)=\sqrt{\frac{2}{a}}e^{\lambda_k^{\ox} T}\sin(\frac{k\pi x}{a})$. We also see that $\big\{\sqrt{\frac{2}{a}}e^{\lambda_k^{\ox} T}\sin(\frac{k\pi x}{a})\big\}_{k\in \N}\subset \mathcal{R}(T,v_0),$ where the reachable space $\mathcal{R}$ is defined by
\begin{align*}
	\mathcal{R}(T,v_0)=\big\{v(T), v \text{ is the solution of } \eqref{KS-oned_cyl} \text{ with } q\in L^2(0,T)  \big\}.
\end{align*}
As the family $\big\{\sqrt{\frac{2}{a}}\sin(\frac{k\pi x}{a})\big\}_{k\in \N}$ is dense in $L^2(\ox)$, we have $\mathcal{R}(T,v_0)$ is dense in $L^2(\ox)$ for any $v_0\in L^2(\ox).$ This implies that \eqref{KS-oned_cyl} is approximately controllable in time $T>0.$ It is immediate from \Cref{aprx} that {$\nu\notin \mathcal{N}_j$}. 
\end{proof}
\subsection{Null controllability}
In the framework of parabolic control theory, the existence of bi-orthogonal families to the family of exponential functions in $L^2(0, T)$ has been extensively studied, from the pioneer work \cite{Fattorini-Russell-1} up to the very recent developments. In this paper, we use \cite[Theorem V.6.43]{Boy23} (which is similar to \cite[Theorem 1.5]{AB2014} but with a more general set of assumptions) to establish the following result.

\begin{theorem}\label{biorthogonal}
	Let $\Lambda$ be a collection of positive real numbers satisfying the following conditions:
	\begin{itemize}
		\item there exist $\theta\in (0,1)$ and $\kappa_1>0$ such that the following asymptotic property holds
		\begin{equation}\label{counting}
				\hspace{2cm}N(r)\leq \kappa_1 r^{\theta},
		\end{equation}
		
		where ${N}$ is the counting function associated with the sequence $\Lambda $ defined by 
		\begin{equation}\label{counting fn}
			{N}(r)= \# \ens{ \lambda\in \Lambda :\, \vert \lambda\vert \leq r}, \quad \forall r>0.
		\end{equation}
		\item $\Lambda$ satisfies the following gap condition: there exists $\rho>0$ such that:
		\begin{equation*}
			|\lambda-\mu|\geq \rho, \quad \forall \lambda\neq \mu \in \Lambda.
		\end{equation*}
	\end{itemize}
	Then for any $T>0$ there exists a family $\{q_{\lambda,T}\}_{\lambda\in \Lambda}$ in $L^2(0,T)$ satisfying
	\begin{equation*}
		\int_{0}^{T}e^{-\mu t} q_{\lambda,T}(t) dt =\delta_{\lambda,\mu}, \quad \forall \lambda,\mu\in \Lambda
	\end{equation*} 
	with the following estimate
	\begin{equation*}
		\norm{q_{\lambda,T}}_{L^2(0,T)}\leq Ke^{K \lambda^{\theta}+KT^{-\frac{\theta}{1-\theta}}} \quad \forall \lambda\in \Lambda,
	\end{equation*}
	where $K(\theta, \rho, \kappa_1)>0$ does not depend on $T.$
\end{theorem}
\begin{proof}
The proof can be found in Theorem V.6.43, \cite{Boy23}.
\end{proof}

\begin{lemma}\label{verification}
Consider {$\nu\notin \mathcal{N}_j$}.
Then the collection of the eigenvalues is $\Lambda=\{-\lambda_k^{\ox}, k\in \N\}$ given by \eqref{eigenv} verifies the condition of \Cref{biorthogonal}.
\end{lemma}
\begin{proof}
Let us denote	\begin{equation}\label{n not_cyl}
		n_0=\min\bigg\{j \in \N: \left(2\mu^{\oy}_j-\nu\right)>0\bigg\}.
	\end{equation}  
We verify the conditions of \Cref{biorthogonal} dividing two cases $j\geq n_0$ and $1\leq j<n_0$ separately.

\textit{Case 1.} $j \geq n_0.$
\begin{itemize}[noitemsep, topsep=0.5mm, leftmargin=*]
	\item {\em Positivity condition.}
	Obvious by assumption on $j.$

	\item {\em The gap condition.} We already assume $\nu\notin \mathcal{N}_j.$ Hence it is clear that $\lambda_k^{\ox}\neq \lambda_m^{\ox}$, $\forall \, k, m \in \N$, $k\neq m.$ Moreover, let us assume that $k>m.$ Then we estimate the gap between the eigenvalues as follows:
\begin{align*}
	|\lambda_k^{\ox}- \lambda_m^{\ox}|=\frac{\pi^2}{a^2}(k^2-m^2)\left|\frac{\pi^2}{a^2}(k^2+m^2)+\left(2\mu^{\oy}_j-\nu\right)\right|\geq \rho,
\end{align*}
for some $\rho>0$ independent of $j.$

\item {\em The condition on the counting function.} Using the definition of the counting function we have for $r>0$:
\begin{equation*}
	N(r)=k \text{ if and only if } |-\lambda_k^{\ox}|\leq r, |-\lambda_{k+1}^{\ox}|>r.
\end{equation*}
Let $N(r)=k$ for some $r>0.$ Then $\big|\frac{k^4\pi^4}{a^4}+\left(2\mu^{\oy}_j-\nu\right) \frac{k^2\pi^2}{a^2}\big|<r$ and it readily follows that $k<\frac{a}{\pi}r^{1/4}.$ Thus we have $N(r)<\frac{a}{\pi}r^{1/4}.$ Hence, inequality $\eqref{counting}$ is verified.

\end{itemize}

\textit{Case 2.} $1\leq j<n_0.$ Note that in this case $\nu_j=\left(-\nu+2\mu^{\oy}_j\right)<0.$
\begin{itemize}[noitemsep, topsep=0.5mm, leftmargin=*]
		\item {\em Positivity condition.}
		Without loss of generality, we assume that all the eigenvalues $-\lambda_k^{\ox}$ are positive. Indeed, we can choose some $c_0>0$ such that $-\lambda_k^{\ox}+c_0>0$ for all $-\lambda_k^{\ox} \in \Lambda.$ In
		what follows, an additional factor $e^{c_0T}$ will appear in the estimation of control cost, but
		without any consequences on our overall analysis. 
		\item {\em The gap condition.} Same as Case 1.
		\item {\em The condition on  counting function.}
		Let $N(r)=k$ for some $r>0.$ 
		Then using definition of counting function we have $\big|\frac{k^4\pi^4}{a^4}+\left(2\mu^{\oy}_j-\nu\right) \frac{k^2\pi^2}{a^2}\big|<r.$
		Let us denote $p=\frac{k^2\pi^2}{a^2}$.  So we have from the above $p^2+\nu_jp-r<0$, where $\nu_j=\left(-\nu+2\mu^{\oy}_j\right)<0.$ We consider the function $f(x)=x^2+\nu_jx-r$, $x\geq 1$. $f$ is convex and it has two real roots $\frac{1}{2}(-\nu_j\pm\sqrt{\nu_j^2+4r}).$ Therefore $f<0$ in the interval
		$\left(\frac{1}{2}(-\nu_j-\sqrt{\nu_j^2+4r}), \frac{1}{2}(-\nu_j+\sqrt{\nu_j^2+4r})\right)$. Which implies that $p<\frac{1}{2}(-\nu_j+\sqrt{\nu_j^2+4r})\leq -\nu_j+\sqrt{r} $. 
		By a straightforward computation and noting that the term ${-\nu_j}$ is bounded for $1\leq j<n_0,$ one can deduce that $N(r)<C+\frac{a}{\pi}r^{1/4},$ for some $C>0$, which gives the desired bound for large $r$. For small $r$, we can choose $\tilde r>0$ such that $N(r)=0$ for $r<\tilde r$, and by possibly increasing $C$, the same estimate $N(r)<Cr^{1/4}$ then holds for all $r\ge \tilde r$. Hence, the bound $N(r)<Cr^{1/4}$ is valid uniformly for all $r>0$.
\end{itemize}
\end{proof}

\subsubsection{\textbf{Reduction to the moment problem}}
To apply the method of moments for showing the null controllability of \eqref{KS-oned_cyl}, we use the adjoint system \eqref{adj_cyl} and derive an equivalent criterion for exact controllability. 
\begin{lemma}\label{lemma moment}
	The control system \eqref{KS-oned_cyl} is null controllable in time $T>0$ in the space $L^2(\Omega_x)$ if and only if for any $v_0\in L^2(\Omega_x)$, there exists a function $q\in L^2(0,T)$ such that for any $ \phi_T \in  L^2(\ox)$ the following identity holds
	\begin{align}\label{eq:equivalent m_cyl}
		\ip{v_0}{\phi(0,\cdot)}_{L^2(\Omega_x)}=\int_{0}^{T} q(t) \phi_x(t,0) dt,
	\end{align}
	where $ \phi$ is the solution of the adjoint system \eqref{adj_cyl}.
\end{lemma}
\begin{proof} At first we consider the terminal data $\phi_T$ and the boundary control $q$ are smooth enough. 
	Taking the inner product of \eqref{KS-oned_cyl} with $\phi$  in $L^2(\Omega_x)$, where $\phi$ is the solution of the adjoint equation \eqref{adj_cyl}, we have
	\begin{align}\label{equivalent m_cyl}
		-\ip{v(T,\cdot)}{\phi_T}_{L^2(\Omega_x)}+	\ip{v_0}{\phi(0,\cdot)}_{L^2(\Omega_x)}=\int_{0}^{T} q(t) \phi_x(t,0) dt.
	\end{align}	
	By density argument we can prove the identity taking $ \phi_T \in  L^2(\ox)$ and  $q \in L^2(0,T)$.
	
	If \eqref{equivalent m_cyl} holds, then $\ip{v(T)}{\phi_T}=0$, $\forall \phi_T \in L^2(\ox)$. Hence $v(T)=0.$ Conversely, if there exists a control $q$ such that the system \eqref{KS-oned_cyl} is exactly controllable, then from \eqref{equivalent m_cyl} one can conclude \eqref{eq:equivalent m_cyl}.
\end{proof}

{
\subsubsection{\textbf{Proof of \Cref{null control 1d_cyl}}}\label{sec:teo16}
Our next task is to convert the above identity into a sequential problem by using the orthonormal eigenbasis $\{\Psi^{\ox}_k\}_{k\in \N}$. Let us consider $\phi_T=\Psi^{\ox}_k$. As $\lambda^{\ox}_k$ is an eigenvalue of the operator $A^*$ and $A$ is self-adjoint, the solution of the adjoint problem \eqref{adj_cyl} becomes
\begin{equation}\label{eq:sigma_m_cyl}
	\phi(t,x)= e^{\lambda^{\ox}_k(T-t)}\Psi^{\ox}_k(x).
\end{equation}
Plugging \eqref{eq:sigma_m_cyl} into \eqref{eq:equivalent m_cyl}, we obtain
\begin{align}
	\nonumber
	e^{\lambda^{\ox}_k T}\ip{u_0}{\Psi^{\ox}_k}_{L^2(\Omega_x)}
	=\sqrt{\frac{2}{a}}\frac{k\pi}{a}
	\int_{0}^{T}q(t) e^{\lambda^{\ox}_k(T-t)}dt.
\end{align}
Using the change of variable $t\mapsto T-t$ and denoting $h(t)=q(T-t)$, we finally have
\begin{align}\label{moment1}
	e^{\lambda^{\ox}_k T}\ip{u_0}{\Psi^{\ox}_k}_{L^2(\Omega_x)}
	=\sqrt{\frac{2}{a}}\frac{k\pi}{a}
	\int_{0}^{T}h(t) e^{\lambda^{\ox}_kt}dt.
\end{align}
Thanks to \Cref{biorthogonal}, there exists a family $\{q_{m,T}\}_{m\in \N}$ in $L^2(0,T)$ such that
\begin{equation}\label{biorth}
	\int_{0}^{T}e^{\lambda^{\ox}_k t} q_{m,T}(t) dt
	=\delta_{k,m}, \quad \forall k,m\in \N,
\end{equation}
with the estimates
\begin{equation}\label{conest1}
\begin{split}
	&\norm{q_{k,T}}_{L^2(0,T)}
	\leq K e^{K(-\lambda^{\ox}_k)^{1/4}+KT^{-1/3}},
	\quad \forall k \in \N, \quad j \geq n_0, \\
	&\norm{q_{k,T}}_{L^2(0,T)}
	\leq K_1 e^{K_1(-\lambda^{\ox}_k+c_0)^{1/4}+K_1T^{-1/3}},
	\quad \forall k \in \N, \quad 1\leq j < n_0,
	\end{split}
\end{equation}
for some positive constants $K$ and $K_1$ independent of $T$ and $j$. In the second estimate, the uniformity with respect to $j$ follows from the fact that there are only finitely many indices $1\leq j<n_0$.

Let us define the control function $h$ by
\begin{equation}\label{exp con}
	h(t)=\sum_{k=1}^{\infty}
	\frac{a}{\sqrt{2}k\pi}
	e^{\lambda^{\ox}_k T}
	\ip{u_0}{\Psi^{\ox}_k}_{L^2(\Omega_x)}
	q_{k,T}(t),
\end{equation}
where we have taken $\theta=1/4$ in \Cref{biorthogonal}. Clearly, this control satisfies the moment problem \eqref{moment1}. It remains to show that $h\in L^2(0,T)$ and to estimate its norm uniformly with respect to $j$.

Let us first assume that $j\geq n_0$ and set $L_{k,j}=-\lambda^{\ox}_k>0$. Combining the factor $e^{\lambda^{\ox}_kT}$ in \eqref{exp con} with the first estimate in \eqref{conest1}, we obtain
\begin{align*}
	e^{\lambda^{\ox}_kT}\norm{q_{k,T}}_{L^2(0,T)}
	&\leq
	K e^{-L_{k,j}T+K L_{k,j}^{1/4}+KT^{-1/3}}.
\end{align*}
By Young's inequality,	$K L_{k,j}^{1/4}=K L_{k,j}^{1/4} T^{1/4}T^{-1/4}
	\leq \frac{T}{4}L_{k,j}+\frac{C}{T^{1/3}}$, and hence
\begin{align}\label{uniform_biorth_high}
	e^{\lambda^{\ox}_kT}\norm{q_{k,T}}_{L^2(0,T)}
	&\leq
	Ce^{C/T^{1/3}}e^{-\frac{3T}{4}L_{k,j}}
	\leq Ce^{C/T^{1/3}}.
\end{align}
Notice that this estimate is uniform with respect to both $k$ and $j$.

When $1\leq j<n_0$, set $L_{k,j}=-\lambda^{\ox}_k+c_0>0$. Since $\lambda^{\ox}_k=c_0-L_{k,j}$, the second estimate in \eqref{conest1} gives
\begin{align*}
	e^{\lambda^{\ox}_kT}\norm{q_{k,T}}_{L^2(0,T)}
	&\leq
	K_1 e^{c_0T-L_{k,j}T+K_1L_{k,j}^{1/4}+K_1T^{-1/3}}.
\end{align*}
Applying Young's inequality once again, we obtain, for $T<1$,
\begin{align}\label{uniform_biorth_low}
	e^{\lambda^{\ox}_kT}\norm{q_{k,T}}_{L^2(0,T)}
	&\leq
	Ce^{C/T^{1/3}}e^{-\frac{3T}{4}L_{k,j}}
	\leq Ce^{C/T^{1/3}}.
\end{align}
where $C$ is independent of $k$, $j$, and $T$.

Using \eqref{exp con}, \eqref{uniform_biorth_high}, and \eqref{uniform_biorth_low}, followed by the Cauchy--Schwarz inequality, we have
\begin{align*}
	\norm{h}_{L^2(0,T)}
	&\leq
	Ce^{C/T^{1/3}}
	\sum_{k=1}^{\infty}
	\frac{1}{k}
	\left|
	\ip{u_0}{\Psi^{\ox}_k}_{L^2(\Omega_x)}
	\right|\\
	&\leq
	Ce^{C/T^{1/3}}
	\left(\sum_{k=1}^{\infty}\frac{1}{k^2}\right)^{1/2}
	\left(
	\sum_{k=1}^{\infty}
	\left|
	\ip{u_0}{\Psi^{\ox}_k}_{L^2(\Omega_x)}
	\right|^2
	\right)^{1/2}.
\end{align*}
Finally, Parseval's identity yields
\begin{equation*}
	\norm{h}_{L^2(0,T)}
	\leq
	Ce^{C/T^{1/3}}
	\norm{u_0}_{L^2(\Omega_x)},
\end{equation*}
where the constant $C>0$ is independent of $T$, $j$, and $u_0$.

The case $T\geq1$ can be reduced to the previous one. Indeed, any control defined on $(0,1/2)$ can be extended by zero to a control on $(0,T)$, and the corresponding estimate follows from the fact that the control cost decreases with respect to time. This ends the proof of \Cref{null control 1d_cyl}.
}

\begin{remark}\label{stable matrix}
	Let us consider the following perturbed system of \eqref{KS-oned_cyl}:
	\begin{equation}\label{KS-oned 1_cyl}
		\begin{cases}
			\partial_t \tl v+\partial_{x}^4 \tl v+\left(\nu-2\mu^{\oy}_j\right)\partial_{x}^2 \tl v+\left((\mu^{\oy}_j)^2-\nu\mu^{\oy}_j\right)\tl v=0,\\
			\tl v(t,0)=\tl v(t,a)=0, \quad \partial_{x}^2 \tl v(t,0)=\tl q(t), \;\; \partial_{x}^2 \tl v(t,a)=0,\\
			\tl v(0,x)=v_0(x).
		\end{cases}
	\end{equation}
	The null controllability result for \eqref{KS-oned_cyl}, namely
	\Cref{null control 1d_cyl}, implies the controllability of system
	\eqref{KS-oned 1_cyl}. Indeed, let us consider the change of variable $\tl v(t,x)
		=
		v(t,x)e^{-\left((\mu^{\oy}_j)^2-\nu\mu^{\oy}_j\right)t}$. Then $\tl q(t)
		=
		e^{-\left((\mu^{\oy}_j)^2-\nu\mu^{\oy}_j\right)t}q(t)$ will be the new control for \eqref{KS-oned 1_cyl}. Also note that $\forall j\geq n_0$, where $n_0$ is defined in \eqref{n not_cyl},
	the quantity $(\mu^{\oy}_j)^2-\nu\mu^{\oy}_j$ is nonnegative, and hence $e^{-\left((\mu^{\oy}_j)^2-\nu\mu^{\oy}_j\right)t}\leq1$. For $j<n_0$, there are only finitely many indices, and the additional
	factor in the control cost is bounded by $e^{c_0T}$. Thus, the uniform control cost estimate of \Cref{null control 1d_cyl}
	also holds for the perturbed system \eqref{KS-oned 1_cyl}.
\end{remark}

\section{Boundary controllability of the KS equation on cylindrical domains}\label{sec:cylinder}
In this section, we aim to establish our main multi-dimensional result, namely \Cref{Ln_KSE}. As outlined in \Cref{st}, the proof relies on using the geometric structure of the problem along with the one-dimensional controllability result presented in \Cref{null control 1d_cyl}. To achieve this, we adopt the Lebeau--Robbiano strategy, following the approach developed in \cite{AB2014}. We denote by $(\lambda^{\oy}_k,\Psi^{\oy}_k)$ the {eigenpair} of the operator
\begin{equation*}
	\begin{cases}
		\mathcal{D}(\tl A)=\{u\in H^4(\Omega_y):u=\Delta u=0\text{ on }\partial\Omega_y\},\\
		\tl Au=-\Delta^2u-\nu\Delta u,\quad \forall u\in\mathcal D(\tl A).
	\end{cases}
\end{equation*}
Therefore, $\lambda^{\oy}_k=-(\mu^{\oy}_k)^2+\nu\mu^{\oy}_k$, where $(\mu^{\oy}_k,\Psi^{\oy}_k)$ is the {eigenpair} of the Dirichlet Laplacian. 
We introduce the closed {subspace} of $L^2(\Omega)$
\begin{equation*}
	E_J=\bigg\{\sum\limits_{j=1}^{J}\ip{u}{\Psi^{\oy}_j}_{L^2(\oy)}\Psi^{\oy}_j : u\in L^2(\Omega)\bigg\}\subset L^2(\Omega), \quad J\geq1,
\end{equation*}
where the notation $\sum\limits_{j=1}^{J}\ip{u}{\Psi^{\oy}_j}_{L^2(\oy)}\Psi^{\oy}_j$ means
\begin{equation*}
	(x,y)\mapsto \sum\limits_{j=1}^{J}\ip{u(x,\cdot)}{\Psi^{\oy}_j}_{L^2(\oy)}\Psi^{\oy}_j(y).
\end{equation*}
We will use the following decomposition result taken from \cite[Lemma 2.1]{AB2014}.
\begin{lemma}\label{normu}
	Any function $u\in L^2(\Omega)$ has the following representation:
	\begin{equation*}
		u=\sum\limits_{j=1}^{\infty}\ip{u}{\Psi^{\oy}_j}_{L^2(\oy)}\Psi^{\oy}_j.
	\end{equation*}
\end{lemma}

\subsection{Partial Observability}
Let us recall that $\mu_k^{\Omega_y}$ are the eigenvalues of the Dirichlet Laplacian in the domain $\Omega_y$. Using the Lebeau--Robbiano spectral inequality \eqref{thm:spec_ineq_degen}, we have, for any $K\in\N$,
\begin{equation}\label{lr_cyl}
	\sum_{k=1}^{K}|a_k|^2\leq C e^{C\sqrt{\mu_K^{\Omega_y}}}\int_{\omega}\left|\sum_{k=1}^{K}a_k \Psi^{\oy}_k(y)\right|^2 dy.
\end{equation}
We denote by $\Pi_{E_J}$ the orthogonal projection in $L^2(\Omega)$ onto $E_J$. A crucial aspect of applying the Lebeau--Robbiano strategy lies in estimating the cost associated with the following partial observability inequality on the finite-dimensional approximation spaces. This estimate plays a central role during the active control phase.
\begin{proposition}\label{par obs_cyl}
	{Let us assume that $\nu\notin\mathcal N$}
	and that $\oy$ is of class $C^2$. Then, for every $\sigma_T\in E_J$, we have the following partial observability inequality
	\begin{equation*}
		\norm{\Pi_{E_J}\sigma(0)}_{L^2(\Omega)}^2
		\leq C {(C_T^{\Omega_x})^2}
		e^{C\sqrt{\mu_J^{\Omega_y}}}
		\int_{0}^{T}\int_{\omega}|\pa_x\sigma(t,0,y)|^2 dy\,dt,
	\end{equation*}
	where $\sigma$ is the solution of the adjoint system \eqref{lin_NKS_adj} and $C_T^{\Omega_x}=Ce^{C/T^{1/3}}$.
\end{proposition}
\begin{proof}
	Note that $\nu\notin\bigcup_{j=1}^{\infty}\mathcal{N}_j$. Consequently, $\nu\notin\mathcal{N}_j$ for every $j\in\mathbb{N}$. Therefore, by \Cref{null control 1d_cyl} and \Cref{stable matrix}, system \eqref{KS-oned 1_cyl} is boundary null controllable with control cost $C_T^{\Omega_x}=Ce^{C/T^{1/3}}$, where the constant is independent of $j$.
	Furthermore, in view of \Cref{normu}, let $\sigma_T=\sum_{j=1}^{J}\sigma_T^{j}(x)\Psi_j^{\Omega_y}(y)$, for some $\sigma_T^{j}\in L^2(\Omega_x)$. If we write
	\begin{equation}\label{dec}
		\sigma(t,x,y)=\sum_{j=1}^{J}\sigma^{j}(t,x)\Psi_j^{\Omega_y}(y),
	\end{equation}
	then the functions $\sigma^{j}$ satisfy the following system in $(0,T)\times\Omega_x$
	\begin{equation}\label{adj2}
		\begin{cases}
			-\partial_t \sigma^j+\partial_{x}^4 \sigma^j+\left(\nu-2\mu^{\oy}_j\right)\partial_{x}^2 \sigma^j+\left((\mu^{\oy}_j)^2-\nu\mu^{\oy}_j\right)\sigma^j=0,\\
			\sigma^j(t,0)=\sigma^j(t,a)=0, \quad \partial_{x}^2 \sigma^j(t,0)=\partial_{x}^2 \sigma^j(t,a)=0,\\
			\sigma^j(T,x)=\sigma^j_T(x).
		\end{cases}
	\end{equation}
	It is clear that the above system is the adjoint system of the one-dimensional KS equation \eqref{KS-oned 1_cyl}.
	As system \eqref{KS-oned 1_cyl} is null controllable, its adjoint system satisfies the observability inequality
	\begin{equation}\label{obs1_cyl}
		\norm{\sigma^j(0)}^2_{L^2(\ox)}
		\leq (C_T^{\Omega_x})^2
		\int_{0}^{T}|\sigma_x^j(t,0)|^2 dt.
	\end{equation}
	Note that $\Pi_{E_J}\sigma(0)=\sigma(0)$. From the expression of $\sigma$, we can write
	\begin{align*}
		\norm{\sigma(0)}^2_{L^2(\Omega)}
		=\sum_{j=1}^{J}\norm{\sigma^j(0)}^2_{L^2(\ox)}.
	\end{align*}
	Thanks to \eqref{obs1_cyl}, we further have
	\begin{align}\label{ob_cyl}
		\norm{\sigma(0)}^2_{L^2(\Omega)}
		\leq (C_T^{\Omega_x})^2
		\int_{0}^{T}\sum_{j=1}^{J}|\sigma_x^j(t,0)|^2 dt.
	\end{align}
	Let us set $a_j=\sigma^j_x(t,0)$ in the Lebeau--Robbiano spectral inequality \eqref{lr_cyl}, that is,
	\begin{equation}\label{ob1_cyl}
		\sum_{j=1}^{J}|\sigma^j_x(t,0)|^2
		\leq C e^{C\sqrt{\mu_J^{\Omega_y}}}
		\int_{\omega}\left|\sum_{j=1}^{J}\sigma^j_x(t,0)\Psi^{\oy}_j(y)\right|^2 dy.
	\end{equation}
	Combining \eqref{ob_cyl} and \eqref{ob1_cyl}, we obtain
	\begin{align*}
		\norm{\sigma(0)}^2_{L^2(\Omega)}
		\leq C (C_T^{\Omega_x})^2 e^{C\sqrt{\mu_J^{\Omega_y}}}
		\int_{0}^{T}\int_{\omega}
		\left|\sum_{j=1}^{J}\sigma^j_x(t,0)\Psi^{\oy}_j(y)\right|^2 dy\,dt,
	\end{align*}
	which gives the desired result.
\end{proof}

Next, by classical duality arguments, we can prove the following result.
\begin{proposition}\label{dual}
	Let $T>0$. There exists $C_T>0$ such that the following two properties are equivalent:
	\begin{itemize}
		\item For every $u_0\in E_J$, there exists a control $q\in L^2(0,T;L^2(\Gamma))$ such that
		\begin{equation*}
			\begin{cases}
				\Pi_{E_J}u(T)=0,\\
				\norm{q}_{L^2(0,T;L^2(\Gamma))}\leq C_T\norm{u_0}_{L^2(\Omega)},
			\end{cases}
		\end{equation*}
		where $u$ is the solution of \eqref{lin_NKS_cyl}.
		\item For every $\sigma_T\in E_J$, the solution $\sigma$ of the adjoint system \eqref{lin_NKS_adj} satisfies
		\begin{equation*}
			\norm{\Pi_{E_J}\sigma(0)}_{L^2(\Omega)}^2
			\leq (C_T)^2
			\int_{0}^{T}\int_{\omega}|\pa_x\sigma(t,0,y)|^2 dy\,dt.
		\end{equation*}
	\end{itemize}
\end{proposition}
The above results, \Cref{par obs_cyl} and \Cref{dual}, imply the following.
\begin{corollary}
	Let $T>0$ and $\nu\notin\mathcal N$. For every $J\geq1$ and $u_0\in E_J$, there exists a control $q_{u_0}\in L^2(0,T;L^2(\Gamma))$ satisfying
	\begin{equation}\label{con cost_cyl}
		\norm{q_{u_0}}_{L^2(0,T;L^2(\Gamma))}
		\leq C C_T^{\Omega_x}e^{C\sqrt{\mu_J^{\Omega_y}}}
		\norm{u_0}_{L^2(\Omega)},
	\end{equation}
	such that the solution of equation \eqref{lin_NKS_cyl} satisfies
	$\Pi_{E_J}u(T)=0$.
\end{corollary}

\subsection{Dissipation along the $\Omega_y$ direction}\label{sec_dis}
Another essential component of the Lebeau–Robbiano strategy is exploiting the system's natural dissipation in the absence of control. In our setting, it is crucial to establish exponential decay in the $\Omega_y$ direction.
\begin{proposition}\label{dis_cyl}
Let us denote
$K_0=\min\{k\in\N:\mu_k^{\oy}>\nu\}$.
Let us consider system \eqref{lin_NKS_cyl} and assume that, on some time interval $(t_0,t_1)$, we have $q=0$. Let $J\geq K_0$ and assume that $\Pi_{E_J}u(t_0)=0$. Then, we have the following estimate:
\begin{equation}\label{dis1_cyl}
\norm{u(t)}_{L^2(\Omega)}
\leq
C e^{\lambda^{\oy}_{J+1}(t-t_0)}
\norm{u(t_0)}_{L^2(\Omega)}
\quad
\forall t\in(t_0,t_1).
\end{equation}
\end{proposition}

\begin{remark}
Note that $\lambda^{\oy}_{J+1}<0$ for every $J\geq K_0$.
\end{remark}

\begin{proof}[Proof of \Cref{dis_cyl}]
Let us express the solution of the KS system \eqref{lin_NKS_cyl} on $(t_0,t_1)$ as
\begin{equation*}
u(t,x,y)=\sum\limits_{j=1}^{\infty}u^j(t,x)\Psi^{\oy}_j(y),
\end{equation*}
where $u^j$ satisfies the following equation in $(t_0,t_1)\times\ox$
\begin{equation}\label{ks2_cyl}
\begin{cases}
\partial_t u^j+\partial_x^4 u^j+\left(\nu-2\mu^{\oy}_j\right)\partial_x^2 u^j
+\left((\mu^{\oy}_j)^2-\nu\mu^{\oy}_j\right)u^j=0,\\
u^j(t,0)=u^j(t,a)=0,\quad
\partial_x^2 u^j(t,0)=\partial_x^2 u^j(t,a)=0.
\end{cases}
\end{equation}
Since $\Pi_{E_J}u(t_0)=0$, we have
\begin{equation}\label{exp_for_u_cyl}
u(t,x,y)=\sum\limits_{j=J+1}^{\infty}u^j(t,x)\Psi^{\oy}_j(y).
\end{equation}
Let us multiply the equation in \eqref{ks2_cyl} by $u^j$ and then integrate over $\ox$. Performing integration by parts, we obtain, for every $j\geq J+1$,
\begin{align*}
\frac{1}{2}\frac{d}{dt}\int_{0}^{a}|u^j(t,x)|^2\,dx
&=
-\int_{0}^{a}|\pa_x^2u^j(t,x)|^2\,dx
+\left(\nu-2\mu^{\oy}_j\right)\int_{0}^{a}|\pa_x u^j(t,x)|^2\,dx \\
&\quad
+\lambda^{\oy}_{j}\int_{0}^{a}|u^j(t,x)|^2\,dx\\
&\leq
\lambda^{\oy}_{J+1}\int_{0}^{a}|u^j(t,x)|^2\,dx,
\end{align*}
where we have used $J\geq K_0$ and
$\lambda^{\oy}_j\leq\lambda^{\oy}_{J+1}<0$ for every $j\geq J+1$.
Therefore, Gronwall's inequality gives
\begin{equation}\label{uj_cyl}
\norm{u^j(t)}_{L^2(\ox)}
\leq
e^{\lambda^{\oy}_{J+1}(t-t_0)}
\norm{u^j(t_0)}_{L^2(\ox)},
\quad
\forall t\in(t_0,t_1).
\end{equation}
Finally, using \eqref{exp_for_u_cyl}, \eqref{uj_cyl}, and the identity
$\norm{u(t)}^2_{L^2(\Omega)}
=\sum_{j=J+1}^{\infty}\norm{u^j(t)}^2_{L^2(\ox)}$,
we obtain the desired result.
\end{proof}

{
	\subsection{The Lebeau-Robbiano approach}\label{LR_md}
In this section, we will prove the boundary null controllability of the linearized KS equation by using the classical approach introduced in \cite{LR}. We divide the time interval $(0,T)$ into an infinite sequence of smaller intervals of appropriate lengths, which will be chosen below along with suitable cutoff frequencies to steer the solution to $0$ in time $T$.

Let $u_0 \in L^2(\Omega)$. We decompose the time interval $[0,T)$ into
\begin{align*}
	[0,T) = \bigcup_{k=0}^{\infty} [a_k, a_{k+1}),
\end{align*}
where $a_0 = 0$, $a_{k+1} = a_k + 2T_k$, and
$T_k = \frac{\alpha}{\beta} 2^{-k\rho}$ for some parameters
$\rho \in \left(0, \frac{1}{N-1}\right)$,
$\alpha = \frac{\beta T}{2}(1 - 2^{-\rho})$, and
$\beta \in \mathbb{N}$, so that
$2\sum_{k=0}^{\infty} T_k = T$. Furthermore, we set the frequency sequence $\gamma_k=\left\lceil\beta^{\frac{N-1}{3}}2^k\right\rceil$, where $\beta \in \mathbb{N}$ is chosen large enough so that
$\gamma_k>K_0$ for all $k\geq0$, with $K_0$ as defined in
\Cref{dis_cyl}. The precise choice of $\beta$ will be specified below.

Next, for all $k\geq 0$, we construct a control $q$ and the solution $u$ of system \eqref{lin_NKS_cyl} by induction as follows:
\begin{equation*}
	q(t)=\begin{cases}
		q\left({\Pi_{E_{\gamma_k}}u(a_k)}\right)(t)
		& \text{if } t \in (a_k, a_k+T_k),\\
		0
		& \text{if } t\in (a_k+T_k, a_{k+1}).
	\end{cases}
\end{equation*}
Our main goal is to show that the control
$q\in L^2(0,T; L^2(\Gamma))$ steers the solution $u$ to rest in time $T$.

\subsubsection{\bf Estimate on the interval $[a_k, a_k+T_k]$}
Thanks to the continuity estimate \eqref{eq:apriori_full} with $f=0$, and since $T_k\leq T$, we have
\begin{equation}\label{uak_cyl}
	\norm{u(a_k+T_k)}_{L^2(\Omega)}
	\leq C\left(
	\norm{u(a_k)}_{L^2(\Omega)}
	+\norm{q}_{L^2(a_k,a_k+T_k;L^2(\Gamma))}
	\right).
\end{equation}
Using the estimate \eqref{con cost_cyl} of the control, we have
\begin{align*}
	\norm{q}_{L^2(a_k,a_k+T_k;L^2(\Gamma))}
	\leq
	C C_{T_k}^{\Omega_x}
	e^{C\sqrt{\mu_{\gamma_k}^{\Omega_y}}}
	\norm{\Pi_{E_{\gamma_k}}u(a_k)}_{L^2(\Omega)}.
\end{align*}
As
$\|\Pi_{E_{\gamma_k}}\|_{\mathcal{L}_{L^2(\Omega)}}\leq1$,
and using the uniform estimate \eqref{cost_cyl} for the control cost
$C^{\ox}_{T_k}$, the above inequality can be written in the form
\begin{align*}
	\norm{q}_{L^2(a_k,a_k+T_k;L^2(\Gamma))}
	\leq
	Ce^{C\left(\frac{1}{T_k^{1/3}}
	+\sqrt{\mu_{\gamma_k}^{\Omega_y}}\right)}
	\norm{u(a_k)}_{L^2(\Omega)}.
\end{align*}
By Weyl's law and the definition of $\gamma_k$, we have
\begin{align*}
	\frac{1}{T_k^{1/3}}
	&=\left(\frac{\beta}{\alpha}\right)^{1/3}
	2^{\frac{k\rho}{3}}
	\leq
	C\beta^{1/3}2^{\frac{k}{N-1}},\\
	\sqrt{\mu_{\gamma_k}^{\Omega_y}}
	&\leq
	C\gamma_k^{\frac{1}{N-1}}
	\leq
	C\beta^{1/3}2^{\frac{k}{N-1}}.
\end{align*}
Therefore,
\begin{align}\label{control est lr_cyl}
	\norm{q}_{L^2(a_k,a_k+T_k;L^2(\Gamma))}
	\leq
	Ce^{C\beta^{1/3}2^{\frac{k}{N-1}}}
	\norm{u(a_k)}_{L^2(\Omega)}.
\end{align}
Thanks to \eqref{uak_cyl}, we obtain
\begin{equation*}
	\norm{u(a_k+T_k)}_{L^2(\Omega)}
	\leq \!
	C\left(1+e^{C\beta^{1/3}2^{\frac{k}{N-1}}}\right)
	\norm{u(a_k)}_{L^2(\Omega)}
	\leq \!
	Ce^{C\beta^{1/3}2^{\frac{k}{N-1}}}
	\norm{u(a_k)}_{L^2(\Omega)}.
\end{equation*}

\subsubsection{\bf Estimate on the interval $[a_k+T_k, a_{k+1}]$}
Since $\Pi_{E_{\gamma_k}}u(a_k+T_k)=0$, using the dissipation result \eqref{dis1_cyl}, we have
\begin{equation*}
	\norm{u(a_{k+1})}_{L^2(\Omega)}
	\leq
	Ce^{\lambda^{\oy}_{\gamma_k+1}T_k}
	\norm{u(a_k+T_k)}_{L^2(\Omega)}.
\end{equation*}

\subsubsection{\bf Final estimate}
Combining the last two estimates, we obtain
\begin{equation*}
	\norm{u(a_{k+1})}_{L^2(\Omega)}
	\leq
	Ce^{\lambda^{\oy}_{\gamma_k+1}T_k}
	e^{C\beta^{1/3}2^{\frac{k}{N-1}}}
	\norm{u(a_k)}_{L^2(\Omega)}.
\end{equation*}
By induction, we have
\begin{equation*}
	\norm{u(a_{k+1})}_{L^2(\Omega)}
	\leq
	Ce^{\sum_{p=0}^{k}\left(
	\lambda^{\oy}_{\gamma_p+1}T_p
	+C\beta^{1/3}2^{\frac{p}{N-1}}
	\right)}
	\norm{u_0}_{L^2(\Omega)}.
\end{equation*}
Note that
\begin{align*}
	-\lambda^{\oy}_{\gamma_p+1}
	&=
	(\mu^{\oy}_{\gamma_p+1})^2
	-\nu\mu^{\oy}_{\gamma_p+1}\geq
	C_1(\mu^{\oy}_{\gamma_p+1})^2
	\geq
	C_2\gamma_p^{\frac{4}{N-1}}
	\geq
	C_3\beta^{4/3}2^{\frac{4p}{N-1}}.
\end{align*}
Consequently,
\begin{align*}
	\lambda^{\oy}_{\gamma_p+1}T_p
	&=
	\frac{\alpha}{\beta}2^{-p\rho}
	\lambda^{\oy}_{\gamma_p+1}\leq
	-C_4\beta^{1/3}
	2^{p\left(\frac{4}{N-1}-\rho\right)}.
\end{align*}
We thus obtain
\begin{equation*}
	\norm{u(a_{k+1})}_{L^2(\Omega)}
	\leq
	Ce^{\sum_{p=0}^{k}\left(
	-C_4\beta^{1/3}
	2^{p\left(\frac{4}{N-1}-\rho\right)}
	+C\beta^{1/3}2^{\frac{p}{N-1}}
	\right)}
	\norm{u_0}_{L^2(\Omega)}.
\end{equation*}
Since $\frac{4}{N-1}-\rho>\frac{1}{N-1}$, there exists $l_0\in\N$ such that
$
	-C_4\beta^{1/3}
	2^{p\left(\frac{4}{N-1}-\rho\right)}
	+C\beta^{1/3}2^{\frac{p}{N-1}}
	\leq
	-C_5\beta^{1/3}
	2^{p\left(\frac{4}{N-1}-\rho\right)}
$
for all $p\geq l_0$. Therefore, for all $k>l_0$, we have
\begin{equation*}
\begin{split}
	\sum_{p=0}^{k}\left(
	-C_4\beta^{1/3}
	2^{p\left(\frac{4}{N-1}-\rho\right)}
	+C\beta^{1/3}2^{\frac{p}{N-1}}
	\right) & \leq
	C'\beta^{1/3}
	-C_5\beta^{1/3}
	\sum_{p=l_0}^{k}
	2^{p\left(\frac{4}{N-1}-\rho\right)}\\
	&\leq
	C'\beta^{1/3}
	-C''\beta^{1/3}
	2^{k\left(\frac{4}{N-1}-\rho\right)}.
\end{split}
\end{equation*}
It follows that
\begin{equation}\label{u at time ak_cyl}
	\norm{u(a_{k+1})}_{L^2(\Omega)}
	\leq
	Ce^{C\beta^{1/3}
	-C\beta^{1/3}2^{k\left(\frac{4}{N-1}-\rho\right)}}
	\norm{u_0}_{L^2(\Omega)}.
\end{equation}

\subsubsection{\bf Control function}
Estimates \eqref{control est lr_cyl} and \eqref{u at time ak_cyl} show that
$q$ belongs to $L^2(0,T;L^2(\Gamma))$. Indeed,
\begin{align*}
	\norm{q}_{L^2(0,T;L^2(\Gamma))}
	&\leq
	\sum_{k=0}^{\infty}
	\norm{q}_{L^2(a_k,a_k+T_k;L^2(\Gamma))}\\
	&\leq
	Ce^{C\beta^{1/3}}\norm{u_0}_{L^2(\Omega)}\\
	&\quad+
	C\sum_{k=l_0+1}^{\infty}
	e^{C\beta^{1/3}
	-C_1\beta^{1/3}
	2^{(k-1)\left(\frac{4}{N-1}-\rho\right)}
	+C\beta^{1/3}2^{\frac{k}{N-1}}}
	\norm{u_0}_{L^2(\Omega)}\\
	&\leq
	Ce^{C\beta^{1/3}}\norm{u_0}_{L^2(\Omega)}
	<\infty.
\end{align*}
We also have the controllability
\begin{equation*}
	\norm{u(T)}_{L^2(\Omega)}
	=
	\lim_{k\to\infty}
	\norm{u(a_{k+1})}_{L^2(\Omega)}
	=0.
\end{equation*}

To obtain the control cost, let us first consider $T<1$. We set $\beta=\left\lceil\frac{\rho_0}{T}\right\rceil$, where $\rho_0$ is sufficiently large and independent of $T$. This choice ensures that $\gamma_k>K_0$ for all $k\geq0$ and that $\alpha$ is bounded from above and below by positive constants independent of $T$. Since $\beta\leq\frac{\rho_0}{T}+1$, we have	$e^{C\beta^{1/3}} \leq Ce^{C/T^{1/3}}$. Therefore, by the preceding analysis, the control for the linear problem \eqref{lin_NKS_cyl} satisfies the control cost estimate $Ce^{C/T^{1/3}}$ for some $C>0$. The case $T\geq1$ can be reduced to the previous one. Indeed, any control defined on $(0,1/2)$ can be extended by zero to a control on $(0,T)$, and the corresponding estimate then follows from the fact that the control cost decreases with respect to time.
}
	 
	 \subsection{Proof of \Cref{Ln_KSE}}\label{sec:proof1p1}
{The sufficiency follows directly from \Cref{LR_md}, where we established the boundary null controllability of \eqref{lin_NKS_cyl} under the assumption
\(\nu\notin\mathcal N\), together with the desired control cost estimate.

It remains to prove the necessity of the condition
\(\nu\notin\mathcal N\). Suppose that the multidimensional system \eqref{lin_NKS_cyl} is boundary null-controllable.
We shall show that this implies the boundary null controllability of the  corresponding one-dimensional system for every \(j\in\mathbb N\).
Indeed, fix \(j\in\mathbb N\) arbitrarily and choose the initial datum of the form
$u_0(x,y)=v_0(x)\Psi_j^{\Omega_y}(y),$
where \(v_0\in L^2(\Omega_x)\).
Let \(u\) be the corresponding controlled solution of \eqref{lin_NKS_cyl} satisfying
\(u(T,\cdot,\cdot)=0\).
Define
$
u^j(t,x)
:=
\int_{\Omega_y}
u(t,x,y)\Psi_j^{\Omega_y}(y)\,dy,
$
and similarly define the projected boundary control
$
q^j(t)
:=
\int_{\omega}
q(t,y)\Psi_j^{\Omega_y}(y)\,dy,
$
where $q\in L^2(0,T;L^2(\Gamma)).$
Using the orthogonality of the eigenfunctions
\(\{\Psi_j^{\Omega_y}\}_{j\in\mathbb N}\)
one obtains, after projecting the multidimensional equation onto
\(\Psi_j^{\Omega_y}\), that \(u^j\) satisfies
\begin{equation*}
\begin{cases}
	\partial_t u^j+\partial_x^4u^j
	+\left(\nu-2\mu_j^{\Omega_y}\right)\partial_x^2u^j
	+\left((\mu_j^{\Omega_y})^2-\nu\mu_j^{\Omega_y}\right)u^j=0,
	\\
	u^j(t,0)=u^j(t,a)=0,
	\partial_x^2u^j(t,0)=q^j(t), \partial_x^2u^j(t,a)=0,
	\\
	u^j(0,x)=v_0(x),
\end{cases}
\end{equation*}
which is precisely the one-dimensional system \eqref{KS-oned_cyl} perturbed by a zero order term. Moreover,
\begin{equation*}
u^j(T,x)
=
\int_{\Omega_y}
u(T,x,y)\Psi_j^{\Omega_y}(y)\,dy
=0,
\end{equation*}
so the projected system is also boundary null-controllable.
Since \(j\in\mathbb N\) was arbitrary, by \Cref{stable matrix} and
\Cref{null control 1d_cyl}, this is possible only if
\(\nu\notin\mathcal N_j\) for every \(j\in\mathbb N\). Consequently,
$
\nu\notin\bigcup_{j=1}^{\infty}\mathcal N_j
=
\mathcal N,
$
which proves the necessity of the condition. This completes the proof of \Cref{Ln_KSE}.}

{
\section{Internal controllability problem}\label{sec:internal} 
In this section, we prove the controllability of the KS equation \eqref{int} with interior control, i.e., \Cref{thm_int}. Let us recall that system \eqref{int} is well-posed in the sense that, for any $u_0 \in L^2(\Omega)$ and $h \in L^2(0,T; L^2(\Omega_y))$, equation \eqref{int} admits a unique solution $u \in C([0,T]; L^2(\Omega))
	\cap L^2(0,T; \mathcal{H}(\Omega))$
and it satisfies the following estimate:
\begin{align*}
	\|u\|_{C([0,T];\Ltwo)}
	+\|u\|_{L^2(0,T;\Hsp)}
	\leq
	Ce^{CT}\left(
	\|u_0\|_{\Ltwo}
	+\|h\|_{L^2(0,T;L^2(\Omega_y))}
	\right).
\end{align*}
Indeed, for any $\eta \in H^1_0(\Omega)$, we have
$\eta,\partial_x\eta\in L^2(0,a;L^2(\Omega_y))$, so 
$\eta\in H^1(0,a;L^2(\Omega_y))$. By the one-dimensional Sobolev embedding $H^1(0,a;L^2(\Omega_y))\hookrightarrow C([0,a];L^2(\Omega_y))$, the point evaluation at $x_0\in(0,a)$ satisfies
\begin{equation*}
	\|\eta(x_0,\cdot)\|_{L^2(\Omega_y)}
	\leq
	C_a\left(
	\|\eta\|_{L^2(\Omega)}
	+\|\partial_x\eta\|_{L^2(\Omega)}
	\right)
	\leq
	C_a\|\eta\|_{H^1(\Omega)}.
\end{equation*}
Consequently, for a.e.\ $t\in(0,T)$, the duality pairing given by
\begin{equation*}
	\left\langle
	\delta_{x_0}\mathbf{1}_{\omega}h(t),\eta
	\right\rangle_{\mathcal{H}(\Omega)',\mathcal{H}(\Omega)}
	=
	\int_\omega h(t,y)\eta(x_0,y)\,dy
\end{equation*}
satisfies $\left|
	\left\langle
	\delta_{x_0}\mathbf{1}_{\omega}h(t),\eta
	\right\rangle
	\right|
	\leq
	C_a\|h(t,\cdot)\|_{L^2(\omega)}
	\|\eta\|_{H^1(\Omega)}$.
This proves that
$\delta_{x_0}\mathbf{1}_{\omega}(y)h(t)\in\mathcal{H}(\Omega)'$.
Standard parabolic theory then ensures the required regularity of the solution.

We now turn our attention to the null controllability problem. We begin by employing the moment method to establish a pointwise null controllability result for the one-dimensional Kuramoto--Sivashinsky equation
\begin{equation}\label{int_oned1}
	\begin{cases}
		\partial_t v+\partial_{x}^4 v
		+\left(\nu-2\mu^{\oy}_j\right)\partial_{x}^2 v
		=\delta_{x_0}h(t)
		& t\in(0,T),\quad x\in\Omega_x,\\
		v(t,0)=v(t,a)=0,\quad
		\partial_x^2v(t,0)=\partial_x^2v(t,a)=0
		& t\in(0,T),\\
		v(0,x)=v_0(x)
		& x\in\Omega_x.
	\end{cases}
\end{equation}
For similar control problems concerning the heat equation, we refer to
\cite{D73} and \cite{L17}. We first state the corresponding well-posedness result for the above system.

\begin{lemma}\label{lm1}
	For any initial state $v_0\in L^2(\Omega_x)$ and any control function
	$h\in L^2(0,T)$, equation \eqref{int_oned1} has a unique solution $v$ in the space $v\in
		L^2(0,T;H^2(\Omega_x)\cap H^1_0(\Omega_x))
		\cap\mc{C}([0,T];L^2(\Omega_x))$. Moreover, we have the following estimate:
	\begin{equation*}
		\norm{v}_{L^2(0,T;H^2(\Omega_x)\cap H^1_0(\Omega_x))}
		+\norm{v}_{\mc{C}([0,T];L^2(\Omega_x))}
		\leq
		C\left(
		\norm{v_0}_{L^2(\Omega_x)}
		+\norm{h}_{L^2(0,T)}
		\right),
	\end{equation*}
	for some constant $C(T)>0$.
\end{lemma}

We start, as usual, our controllability study with the following approximate controllability result.

\begin{proposition}\label{ns1}
	The necessary and sufficient condition for the approximate controllability of \eqref{int_oned1} is
	${x_0}/{a}\in(0,1)\setminus\mathbb{Q}$ and
	$\nu\notin\mathcal{N}_j$.
\end{proposition}

\begin{proof}
	Let us recall that the approximate controllability of system
	\eqref{int_oned1} is equivalent to the following property: every solution $v$ of the corresponding homogeneous equation (i.e. $h\equiv 0$) satisfying
	$v(\cdot,x_0)=0$ is identically zero.

	\noindent
	\textit{Necessary condition:}
	Let $x_0/a\in(0,1)\cap\mathbb{Q}$. Thus, $x_0/a=p/q$ for some natural numbers $p,q$. If we take
	$v_0=\sin\left(\frac{q\pi x}{a}\right)$, then the corresponding solution is $v(t,x)
		=
		e^{\lambda_q^{\ox}t}
		\sin\left(\frac{q\pi x}{a}\right)$.
	Thus, $v(t,x_0)=0$, but $v$ is nonzero. Therefore,
	$x_0/a\in(0,1)\setminus\mathbb{Q}$ is necessary.

	Furthermore, let $\nu\in\mathcal{N}_j$ and
	$x_0/a\in(0,1)\setminus\mathbb{Q}$. Then there exist some
	$k_0\neq l_0$ such that
	$\nu=2\mu_j^{\oy}
		+\pi^2\left(\frac{k_0^2+l_0^2}{a^2}\right)$. Therefore, we have
	\begin{align*}
		\lambda^{\ox}_{k_0}
		=
		-\frac{k_0^4\pi^4}{a^4}
		+\left(\nu-2\mu_j^{\oy}\right)
		\frac{k_0^2\pi^2}{a^2}=
		\frac{k_0^2l_0^2\pi^4}{a^4}
		=
		-\frac{l_0^4\pi^4}{a^4}
		+\left(\nu-2\mu_j^{\oy}\right)
		\frac{l_0^2\pi^2}{a^2}
		=
		\lambda^{\ox}_{l_0}.
	\end{align*}
	Next, observe that
	\begin{equation*}
		v(t,x)
		=
		e^{\lambda^{\ox}_{k_0}t}
		\sin\left(\frac{k_0\pi x}{a}\right)
		-
		\frac{
		\sin\left(\frac{k_0\pi x_0}{a}\right)
		}{
		\sin\left(\frac{l_0\pi x_0}{a}\right)
		}
		e^{\lambda^{\ox}_{k_0}t}
		\sin\left(\frac{l_0\pi x}{a}\right)
	\end{equation*}
	is a nonzero solution of \eqref{int_oned1} (with $h\equiv 0$). This violates the approximate controllability. Hence, $\nu\notin\mathcal{N}_j$ is necessary.

	\noindent
	\textit{Sufficient condition:}
	Let $x_0/a\in(0,1)\setminus\mathbb{Q}$ and
	$\nu\notin\mathcal{N}_j$. Take $v_0=\sum_{k=1}^{\infty}c_k\sin\left(\frac{k\pi x}{a}\right)$.	Then, the solution $v$ of the form $ v(t,x)
		=
		\sum_{k=1}^{\infty}
		c_k e^{\lambda_k^{\ox}t}
		\sin\left(\frac{k\pi x}{a}\right)$. If $v(t,x_0)=0$, then $\sum_{k=1}^{\infty}
		c_k e^{\lambda_k^{\ox}t}
		\sin\left(\frac{k\pi x_0}{a}\right)
		=0$. Following the same argument as in the proof of \Cref{aprx} we conclude the proof.
\end{proof}

\subsection{Null controllability}

Before stating our controllability result, let us first recall the following minimal time:
\begin{align}\label{tj}
	T_0(x_0)
	:=
	\limsup_{k\to+\infty}
	\frac{
	-\log\left(
	\left|\sin\left(\frac{k\pi x_0}{a}\right)\right|
	\right)
	}{
	\frac{k^4\pi^4}{a^4}
	}.
\end{align}
Our goal is to prove the following result.

\begin{theorem}\label{int_nullcontrol}
	Let $T>0$, $\nu\notin\mathcal{N}_j$, and
	$x_0/a\in(0,1)\setminus\mathbb{Q}$. Then, for every
	$v_0\in L^2(\Omega_x)$, there exists a control $h\in L^2(0,T)$ such that system \eqref{int_oned1} satisfies $v(T)=0$ whenever
	$T>T_0(x_0)$. Moreover, system \eqref{int_oned1} is not null controllable when $T<T_0(x_0)$.
\end{theorem}

Recall the adjoint system \eqref{adj_cyl}. The following lemma gives an equivalent criterion for null controllability.

\begin{lemma}\label{moment_pr1}
	Let $T>0$ and $v_0\in L^2(\Omega_x)$ be given. Then, the KS equation
	\eqref{int_oned1} is null controllable at time $T$ by means of a control
	$h\in L^2(0,T)$ if and only if for all $\phi_T\in L^2(\Omega_x)$ the following identity holds
	\begin{equation}\label{equiv_id}
		\int_0^T\phi(t,x_0)h(t)\,dt
		=
		-\int_0^a v_0(x)\phi(0,x)\,dx
	\end{equation}
where $\phi$ is the solution of the adjoint system \eqref{adj_cyl}.
\end{lemma}

Arguing as in the proof of \Cref{null control 1d_cyl}, we can deduce from \eqref{equiv_id} the following moment problem
\begin{equation}\label{moment_prb}
	\int_0^T e^{\lambda_k^{\ox}t}\widetilde h(t)\,dt
	=
	-\frac{
	e^{\lambda_k^{\ox}T}
	}{
	\Psi_k^{\ox}(x_0)
	}
	\ip{v_0}{\Psi_k^{\ox}}_{L^2(\Omega_x)},
	\qquad k\in\mathbb{N}.
\end{equation}
To solve it, we shall use the same family
$\{q_{k,T}\}_{k\in\mathbb{N}}$ constructed in the proof of
\Cref{null control 1d_cyl}, which satisfies
\begin{equation}\label{biorth_int}
	\int_0^T e^{\lambda_l^{\ox}t}q_{k,T}(t)\,dt
	=
	\delta_{k,l},
	\qquad k,l\in\mathbb{N},
\end{equation}
and the estimates
\begin{align}\label{conest_int}
	\norm{q_{k,T}}_{L^2(0,T)}
	&\leq
	Ke^{K(-\lambda_k^{\ox})^{1/4}+KT^{-1/3}},
	&& k\in\mathbb{N},\quad j\geq n_0,
	\\ \label{conest_int_2}
	\norm{q_{k,T}}_{L^2(0,T)}
	&\leq
	K_1e^{K_1(-\lambda_k^{\ox}+c_0)^{1/4}+K_1T^{-1/3}},
	&& k\in\mathbb{N},\quad 1\leq j<n_0.
\end{align}
As before, the constants in \eqref{conest_int}--\eqref{conest_int_2} can be chosen independently of $j$.

We are now in a position to prove \Cref{int_nullcontrol}.

\begin{proof}[Proof of \Cref{int_nullcontrol}]

Let us fix $j\in\mathbb{N}$. We first prove that system
\eqref{int_oned1} is null controllable for all $T>T_0(x_0)$. Define
\begin{equation}\label{exp con1}
	\widetilde h(t)
	=
	-\sum_{k=1}^{\infty}
	\frac{
	e^{\lambda_k^{\ox}T}
	}{
	\Psi_k^{\ox}(x_0)
	}
	\ip{v_0}{\Psi_k^{\ox}}_{L^2(\Omega_x)}
	q_{k,T}(t).
\end{equation}
In view of \eqref{biorth_int}, this control formally satisfies the moment problem \eqref{moment_prb}. It remains to prove that the series converges in
$L^2(0,T)$.

Assume that $T>T_0(x_0)$ and set
$\delta=T-T_0(x_0)>0$. In particular, $T_0(x_0)<\infty$.
Let $A_k=\frac{k^4\pi^4}{a^4}$. For every $\epsilon>0$, the definition
of $T_0(x_0)$ implies that
\begin{equation}\label{sin_est_general}
	\frac{1}{
	\left|
	\sin\left(\frac{k\pi x_0}{a}\right)
	\right|
	}
	\leq
	C_\epsilon e^{(T_0(x_0)+\epsilon)A_k},
	\qquad k\in\mathbb{N}.
\end{equation}
We take $\epsilon=\delta/4$.

Suppose first that $j\geq n_0$. Then define $L_{k,j}
	:=
	-\lambda_k^{\ox}
	=
	A_k+B_{k,j}$ with $B_{k,j}:=(2\mu_j^{\oy}-\nu)\frac{k^2\pi^2}{a^2}$. Note that $B_{k,j}\geq 0$ for all $k\in\mathbb N$ and $j\geq n_0$. Combining the first estimate in \eqref{conest_int} with
\eqref{sin_est_general}, we have
\begin{align*}
	\frac{
	e^{\lambda_k^{\ox}T}
	}{
	|\Psi_k^{\ox}(x_0)|
	}
	\norm{q_{k,T}}_{L^2(0,T)}
	&\leq
	C_\epsilon
	e^{-TL_{k,j}
	+(T_0(x_0)+\epsilon)A_k
	+KL_{k,j}^{1/4}
	+KT^{-1/3}}.
\end{align*}
By Young's inequality, $KL_{k,j}^{1/4}
	\leq
	\frac{\delta}{4}L_{k,j}
	+\frac{C}{\delta^{1/3}}$. Moreover, since $T_0(x_0)+\epsilon=T-3\delta/4$, we have $-\left(T-\frac{\delta}{4}\right)L_{k,j}
	+\left(T-\frac{3\delta}{4}\right)A_k
	=
	-\frac{\delta}{2}A_k
	-\left(T-\frac{\delta}{4}\right)B_{k,j}$. Therefore,
\begin{align}\label{int_weight_high}
	\frac{
	e^{\lambda_k^{\ox}T}
	}{
	|\Psi_k^{\ox}(x_0)|
	}
	\norm{q_{k,T}}_{L^2(0,T)}
	&\leq
	C_{T,\delta}
	e^{-\frac{\delta}{2}A_k}
	e^{-\left(T-\frac{\delta}{4}\right)
	(2\mu_j^{\oy}-\nu)
	\frac{k^2\pi^2}{a^2}}\leq
	C_{T,\delta}e^{-\frac{\delta}{2}A_k}.
\end{align}

Now suppose that $1\leq j<n_0$ and set $L_{k,j}:=-\lambda_k^{\ox}+c_0
	=A_k-D_{k,j}+c_0$, where $D_{k,j}:=
	(\nu-2\mu_j^{\oy})\frac{k^2\pi^2}{a^2}$. In this case $D_{k,j}\geq 0$. Combining \eqref{conest_int_2} and \eqref{sin_est_general}, and using
$\lambda_k^{\ox}=c_0-L_{k,j}$, we obtain
\begin{align*}
	\frac{
	e^{\lambda_k^{\ox}T}
	}{
	|\Psi_k^{\ox}(x_0)|
	}
	\norm{q_{k,T}}_{L^2(0,T)}
	&\leq
	C_\epsilon
	e^{c_0T-TL_{k,j}
	+(T_0(x_0)+\epsilon)A_k
	+K_1L_{k,j}^{1/4}
	+K_1T^{-1/3}}.
\end{align*}
Using Young's inequality as above and recalling that
$T_0(x_0)+\epsilon=T-3\delta/4$, we obtain
\begin{align*}
	\frac{
	e^{\lambda_k^{\ox}T}
	}{
	|\Psi_k^{\ox}(x_0)|
	}
	\norm{q_{k,T}}_{L^2(0,T)}
	&\leq
	C_{T,\delta}
	e^{-\frac{\delta}{2}A_k
	+\left(T-\frac{\delta}{4}\right)D_{k,j}}.
\end{align*}
Since there are only finitely many indices $1\leq j<n_0$, we have
$D_{k,j}\leq Ck^2$ uniformly with respect to $j$. Therefore, $\left(T-\frac{\delta}{4}\right)D_{k,j}
	\leq
	CTk^2
	\leq
	\frac{\delta}{4}A_k+C_{T,\delta}$. Consequently,
\begin{equation}\label{int_weight_low}
	\frac{
	e^{\lambda_k^{\ox}T}
	}{
	|\Psi_k^{\ox}(x_0)|
	}
	\norm{q_{k,T}}_{L^2(0,T)}
	\leq
	C_{T,\delta}e^{-\frac{\delta}{4}A_k}
	\leq
	C_{T,\delta}e^{-C\delta k^4},
	\qquad k\in\mathbb{N},
\end{equation}
where the constant can be chosen uniformly for $1\leq j<n_0$.

Using \eqref{exp con1}, followed by the Cauchy--Schwarz inequality, we obtain
\begin{align*}
	\norm{\widetilde h}_{L^2(0,T)}
	&\leq
	\left(
	\sum_{k=1}^{\infty}
	\left|
	\ip{v_0}{\Psi_k^{\ox}}_{L^2(\Omega_x)}
	\right|^2
	\right)^{1/2}
	\left(
	\sum_{k=1}^{\infty}
	\frac{
	e^{2\lambda_k^{\ox}T}
	}{
	|\Psi_k^{\ox}(x_0)|^2
	}
	\norm{q_{k,T}}_{L^2(0,T)}^2
	\right)^{1/2}.
\end{align*}
Estimates \eqref{int_weight_high} and \eqref{int_weight_low}, together with Parseval's identity, show that
\begin{equation*}
	\norm{\widetilde h}_{L^2(0,T)}
	\leq
	C_{T,T_0(x_0)}
	\norm{v_0}_{L^2(\Omega_x)}.
\end{equation*}
Hence, the series in \eqref{exp con1} converges in $L^2(0,T)$ and defines a control satisfying \eqref{moment_prb}. Therefore, system
\eqref{int_oned1} is null controllable whenever $T>T_0(x_0)$.

We now prove the negative result. Fix $j\in\mathbb{N}$ and suppose, by contradiction, that system \eqref{int_oned1} is null controllable at some time
$T<T_0(x_0)$. Then the corresponding adjoint system satisfies the observability inequality
\begin{equation*}
	\norm{\phi(0,\cdot)}_{L^2(\Omega_x)}^2
	\leq
	C\int_0^T|\phi(t,x_0)|^2\,dt
\end{equation*}
for some $C>0$. Recalling \eqref{eq:sigma_m_cyl}, consider the terminal datum $\phi_T=\Psi_k^{\ox}$. Then, for all sufficiently large $k$, we have $\lambda_k^{\ox}<0$, and the above observability inequality reduces to
\begin{align}\label{int_est1}
	e^{2\lambda_k^{\ox}T}
	\leq
	C
	|\Psi_k^{\ox}(x_0)|^2
	\int_0^T e^{2\lambda_k^{\ox}(T-t)}\,dt \leq
	C
	\sin^2\left(\frac{k\pi x_0}{a}\right).
\end{align}

Assume that $T_0(x_0)<\infty$. By the definition \eqref{tj}, there exists an increasing unbounded sequence
$\{k_n\}_{n\in\mathbb{N}}$ such that
\begin{equation*}
	T_0(x_0)
	=
	\lim_{n\to+\infty}
	\frac{
	-\log\left(
	\left|
	\sin\left(\frac{k_n\pi x_0}{a}\right)
	\right|
	\right)
	}{
	A_{k_n}
	}.
\end{equation*}
For every $\epsilon>0$, there exists $n_\epsilon\in\mathbb{N}$ such that
\begin{equation*}
	\left|
	\sin\left(\frac{k_n\pi x_0}{a}\right)
	\right|
	\leq
	e^{-(T_0(x_0)-\epsilon)A_{k_n}},
	\quad n\geq n_\epsilon.
\end{equation*}
Then from \eqref{int_est1}, we deduce that
\begin{align*}
	1 \leq
	C
	\left|
	\sin\left(\frac{k_n\pi x_0}{a}\right)
	\right|
	e^{-\lambda_{k_n}^{\ox}T} \leq
	C
	e^{-(T_0(x_0)-T-\epsilon)A_{k_n}}
	e^{(2\mu_j^{\oy}-\nu)
	\frac{k_n^2\pi^2}{a^2}T}.
\end{align*}
Choosing $\epsilon=\frac{T_0(x_0)-T}{2}$, we obtain
\begin{equation*}
	1
	\leq
	Ce^{-C_1k_n^4+C_{j}Tk_n^2}.
\end{equation*}
Since $j$ and $T$ are fixed, the right-hand side converges to zero as
$n\to+\infty$, which yields a contradiction.

Finally, assume that $T_0(x_0)=\infty$. By the definition \eqref{tj}, there exists an increasing unbounded sequence
$\{k_n\}_{n\in\mathbb{N}}$ such that
\begin{equation*}
	\lim_{n\to+\infty}
	\frac{
	-\log\left(
	\left|
	\sin\left(\frac{k_n\pi x_0}{a}\right)
	\right|
	\right)
	}{
	A_{k_n}
	}
	=
	\infty.
\end{equation*}
Fix any $T>0$. Then, for all sufficiently large $n$, we have
$\left|
	\sin\left(\frac{k_n\pi x_0}{a}\right)
	\right|
	\leq
	e^{-2TA_{k_n}}$. On the other hand, from \eqref{int_est1} we can deduce that
\begin{equation*}
	1
	\leq
	C
	\left|
	\sin\left(\frac{k_n\pi x_0}{a}\right)
	\right|
	e^{-\lambda_{k_n}^{\ox}T}.
\end{equation*}
Since $-\lambda_{k_n}^{\ox}
	=
	A_{k_n}
	+
	\left(2\mu_j^{\oy}-\nu\right)
	\frac{k_n^2\pi^2}{a^2}$, we obtain, for all sufficiently large $n$,
\begin{equation*}
	1
	\leq
	Ce^{-TA_{k_n}+C_jTk_n^2}.
\end{equation*}
The right-hand side converges to zero as $n\to+\infty$, which yields a contradiction. Hence, system \eqref{int_oned1} is not null controllable at any finite time when $T_0(x_0)=\infty$. The proof of \Cref{int_nullcontrol} is finished.
\end{proof}

\begin{theorem}\label{int_thm_on}
	Let $T>0$, $\nu\notin\mathcal{N}_j$, and
	$x_0/a\in(0,1)\setminus\mathbb{Q}$. Assume, in addition, that
	$x_0/a$ is an algebraic real number of order $d>1$. Then, for every
	$v_0\in L^2(\Omega_x)$, there exists a control $h\in L^2(0,T)$ such that system \eqref{int_oned1} satisfies $v(T)=0$. Moreover, the control satisfies
	\begin{equation*}
		\norm{h}_{L^2(0,T)}
		\leq
		Ce^{C/T^{1/3}}
		\norm{v_0}_{L^2(\Omega_x)},
	\end{equation*}
	where $C>0$ is independent of $T$, $j$, and $v_0$.
\end{theorem}

\begin{proof}
	Since $x_0/a$ is an algebraic irrational number of degree $d>1$,
	Liouville's theorem implies that (see the proof of
	\cite[Proposition 3.1]{S15})
	\begin{equation}\label{liouville_sin}
		\frac{1}{
		\left|
		\sin\left(\frac{k\pi x_0}{a}\right)
		\right|
		}
		\leq
		Ck^{d-1},
		\qquad k\in\mathbb{N}.
	\end{equation}
	In particular, \eqref{tj} gives $T_0(x_0)=0$, and hence the null
	controllability for every $T>0$ follows from
	\Cref{int_nullcontrol}. It remains to estimate the control.

	Let us first assume that $T<1$ and use the control defined by
	\eqref{exp con1}. We proceed as in the estimates
	\eqref{int_weight_high} and \eqref{int_weight_low}, now using
	\eqref{liouville_sin}. More precisely, when $j\geq n_0$, the first
	estimate in \eqref{conest_int} and Young's inequality give $e^{\lambda_k^{\ox}T}
		\norm{q_{k,T}}_{L^2(0,T)}
		\leq
		Ce^{C/T^{1/3}}
		e^{-\frac{3T}{4}(-\lambda_k^{\ox})}$, where $-\lambda_k^{\ox}\geq ck^4$. When $1\leq j<n_0$, the estimate \eqref{conest_int_2} gives $e^{\lambda_k^{\ox}T}
		\norm{q_{k,T}}_{L^2(0,T)}
		\leq
		Ce^{C/T^{1/3}}
		e^{-\frac{3T}{4}(-\lambda_k^{\ox}+c_0)}$. Since there are only finitely many $j$ indices, we have
	$-\lambda_k^{\ox}+c_0\geq ck^4$ uniformly with respect to
	$1\leq j<n_0$. Therefore, in both cases, \eqref{liouville_sin}
	yields
	\begin{equation}\label{alg_weight}
		\frac{
		e^{\lambda_k^{\ox}T}
		}{
		|\Psi_k^{\ox}(x_0)|
		}
		\norm{q_{k,T}}_{L^2(0,T)}
		\leq
		Ce^{C/T^{1/3}}
		k^{d-1}e^{-cTk^4},
		\qquad k\in\mathbb{N},
	\end{equation}
	where $C,c>0$ are independent of $j$.

	Using \eqref{alg_weight} in \eqref{exp con1}, followed by the
	Cauchy--Schwarz inequality and Parseval's identity, we deduce that
	\begin{align*}
		\norm{\widetilde h}_{L^2(0,T)}
		&\leq
		Ce^{C/T^{1/3}}
		\left(
		\sum_{k=1}^{\infty}
		k^{2(d-1)}e^{-cTk^4}
		\right)^{1/2}
		\norm{v_0}_{L^2(\Omega_x)}.
	\end{align*}
	Since
	$\sum_{k=1}^{\infty}k^{2(d-1)}e^{-cTk^4}
	\leq C\left(1+T^{-(2d-1)/4}\right)$, the resulting polynomial
	factor can be absorbed into $e^{C/T^{1/3}}$ for $T<1$. Hence,
	\begin{equation*}
		\norm{h}_{L^2(0,T)}
		=
		\norm{\widetilde h}_{L^2(0,T)}
		\leq
		Ce^{C/T^{1/3}}
		\norm{v_0}_{L^2(\Omega_x)}.
	\end{equation*}

	The case $T\geq1$ can be reduced to the previous one. Indeed, any
	control defined on $(0,1/2)$ can be extended by zero to a control
	on $(0,T)$, and the corresponding estimate follows from the fact
	that the control cost decreases with respect to time.
\end{proof}

\begin{remark}\label{nec_int}
	\Cref{ns1}, together with an argument similar to that in
	\Cref{necessary}, shows that the conditions appearing in
	\Cref{int_nullcontrol}, namely,
	$\nu\notin\mathcal{N}_j$ and
	$x_0/a\in(0,1)\setminus\mathbb{Q}$, are also necessary for null
	controllability at any time $T>T_0(x_0)$.
\end{remark}

\subsection{Proof of \Cref{thm_int}}
We split the proof into two parts.
\begin{itemize}[topsep=0.5mm, leftmargin=*]
	\item[1.]
	Take $\omega=\Omega_y$ and let $T>T_0(x_0)$. We will prove the
	following observability inequality:
	\begin{equation*}
		\norm{\sigma(0)}_{L^2(\Omega)}^2
		\leq
		C\int_0^T
		\norm{\sigma(t,x_0,\cdot)}_{L^2(\Omega_y)}^2\,dt
	\end{equation*}
	for the adjoint system \eqref{lin_NKS_adj}. Let
$\sigma_T=\sum_{j=1}^{\infty}\sigma_T^j(x)\Psi_j^{\oy}(y)$ where $\sigma_T^j\in L^2(\ox)$,
	and write $\sigma(t,x,y)=\sum_{j=1}^{\infty}\sigma^j(t,x)\Psi_j^{\oy}(y)$. Then $\sigma^j$ satisfies the corresponding one-dimensional adjoint
	system \eqref{adj2}. Thanks to \Cref{int_nullcontrol}, system \eqref{int_oned1} is null
	controllable in time $T>T_0(x_0)$, with a control cost uniform with
	respect to $j$. Using this fact and arguing as in
	Remark~\ref{stable matrix}, we deduce that the following perturbed
	system in $(0,T)\times\Omega_x$
	\begin{equation}\label{KS-oned 1_cyl_int}
		\begin{cases}
			\partial_t\widetilde v+\partial_x^4\widetilde v
			+\left(\nu-2\mu_j^{\oy}\right)\partial_x^2\widetilde v
			+\left((\mu_j^{\oy})^2-\nu\mu_j^{\oy}\right)\widetilde v
			=\delta_{x_0}h(t),\\
			\widetilde v(t,0)=\widetilde v(t,a)=0,\quad
			\partial_x^2\widetilde v(t,0)
			=\partial_x^2\widetilde v(t,a)=0,\\
			\widetilde v(0,x)=v_0(x),
		\end{cases}
	\end{equation}
	is also null controllable with a control cost of the form
	\begin{equation*}
		\norm{h}_{L^2(0,T)}
		\leq
		C_T\norm{v_0}_{L^2(\Omega_x)},
	\end{equation*}
	where $C_T>0$ is independent of $j$. Therefore, the corresponding adjoint system \eqref{adj2} satisfies
	the observability inequality
	\begin{equation}\label{obs1_cyl1}
		\norm{\sigma^j(0)}_{L^2(\ox)}^2
		\leq
		C_T^2\int_0^T|\sigma^j(t,x_0)|^2\,dt.
	\end{equation}
	Summing \eqref{obs1_cyl1} with respect to $j$, we obtain
	\begin{align*}
		\norm{\sigma(0)}_{L^2(\Omega)}^2
		&=
		\sum_{j=1}^{\infty}
		\norm{\sigma^j(0)}_{L^2(\ox)}^2\leq
		C_T^2
		\sum_{j=1}^{\infty}
		\int_0^T|\sigma^j(t,x_0)|^2\,dt\\
		&=
		C_T^2
		\int_0^T
		\norm{\sigma(t,x_0,\cdot)}_{L^2(\Omega_y)}^2\,dt.
	\end{align*}
	This shows that equation \eqref{int} is null controllable in time
	$T>T_0(x_0)$.

	Next, note that if the $N$-dimensional system \eqref{int} is null
	controllable at time $T$, then each one-dimensional system
	\eqref{int_oned1} is also null controllable at time $T$. This follows
	by projecting the equation and the control onto each eigenfunction
	$\Psi_j^{\oy}$. Since \eqref{int_oned1} is not null controllable when
	$T<T_0(x_0)$, the same holds for \eqref{int}. This Fourier
	decomposition argument and \Cref{nec_int} also show that $\nu\notin\mathcal{N}$ and $x_0/a\in(0,1)\setminus\mathbb{Q}$ are necessary for the null controllability of \eqref{int} at any
	time $T>T_0(x_0)$.

	\item[2.]
	Let $\omega\subsetneq\Omega_y$. The proof is similar to that of
	\Cref{Ln_KSE}; hence, we provide only a brief sketch. Recall the
	decomposition \eqref{dec} and the adjoint system \eqref{adj2}.
	Thanks to \Cref{int_thm_on} and the same change of variables used
	above, system \eqref{KS-oned 1_cyl_int} is null controllable with
	\begin{equation*}
		\norm{h}_{L^2(0,T)}
		\leq
		C_T\norm{v_0}_{L^2(\Omega_x)},
		\qquad
		C_T\leq Ce^{CT+C/T^{1/3}},
	\end{equation*}
	where the constants are independent of $j$. Therefore, the solution
	of the adjoint system \eqref{adj2} satisfies
	\begin{equation*}
		\norm{\sigma^j(0)}_{L^2(\ox)}^2
		\leq
		C_T^2
		\int_0^T|\sigma^j(t,x_0)|^2\,dt.
	\end{equation*}

	Let $\sigma_T\in E_J$. Setting
	$a_j=\sigma^j(t,x_0)$ in the Lebeau--Robbiano spectral inequality
	\eqref{lr_cyl}, and using the fact that
	$\Pi_{E_J}\sigma(0)=\sigma(0)$, we deduce
	\begin{align*}
		\norm{\sigma(0)}_{L^2(\Omega)}^2
		\leq
		CC_T^2e^{C\sqrt{\mu_J^{\Omega_y}}}
		\int_0^T\int_\omega
		\left|
		\sum_{j=1}^J
		\sigma^j(t,x_0)\Psi_j^{\oy}(y)
		\right|^2\,dy\,dt.
	\end{align*}
	By standard duality, for every $J\geq1$ and $u_0\in E_J$, there
	exists a control $h\in L^2(0,T;L^2(\omega))$ satisfying
	\begin{equation*}
		\norm{h}_{L^2(0,T;L^2(\omega))}
		\leq
		CC_Te^{C\sqrt{\mu_J^{\Omega_y}}}
		\norm{u_0}_{L^2(\Omega)}
	\end{equation*}
	such that the solution of \eqref{int} satisfies
	$\Pi_{E_J}u(T)=0$.

	The remainder of the proof follows the same steps developed in \Cref{sec_dis} and \Cref{LR_md}.
\end{itemize}
}

\section{Local controllability for the nonlinear KS equation}\label{sec:nonlinear}
We present briefly the proof of \Cref{th_nl_KS} since it follows the same steps as in \cite[Theorem 1.4]{TT17} for $N=2 \text{ or }3$. The idea is to employ the source term method \cite{Tucsnak-nonlinear} followed by the Banach fixed point theorem to ensure the aforementioned local null controllability result. Let us first assume the constants  $p>0$, $\gamma>1$ in such a way that 
\begin{align}\label{choice-p_q}
	1<\gamma<\sqrt{2}, \ \ \text{and} \ \ p> \frac{\gamma^2}{2-\gamma^2}.
\end{align}
Let $C$ be the positive constant obtained in the control cost estimate (see \Cref{Ln_KSE}). Let us define the functions 

\begin{equation}\label{def_weight_func}
	\rho_0(t)=\left\{\begin{array}{ll} e^{-\frac{pC}{(\gamma-1)(T-t)}} & t \in [0, T), \\ 0 & t=T, \end{array}\right.   \quad 
	\rho_{\F}(t)= \left\{\begin{array}{ll} e^{-\frac{(1+p)\gamma^2 C}{(\gamma-1)(T-t)}} & t \in [0, T), \\ 0 & t=T. \end{array}\right.
\end{equation} 
Note that the functions $\rho_0$ and $\rho_{\F}$ are continuous and non-increasing in $[0,T]$. 
We define the following weighted spaces
\begin{subequations}
	\begin{align}
		\label{space_F}
		&\F:= \left\{f\in L^2(0,T; \mathcal{H}'(\Omega)) \ \Big| \ \frac{f}{\rho_\F} \in L^2(0,T; \mathcal{H}'(\Omega))  \right\}, \\
		\label{space_Y} 
		&	\Y := \left\{ u\in C([0, T]; L^2(\Omega)) \Big| \frac{u}{\rho_0} \in C([0,T]; L^2(\Omega)) \cap L^2(0,T; \mathcal H(\Omega))\right\}, \\
		\label{space_V}
		&	\V:= \left\{ q\in L^2(0,T;L^2(\Gamma))  \ \Big| \ \frac{q}{\rho_0}\in L^2(0,T;L^2(\Gamma))   \right\}.
	\end{align}
\end{subequations}
Let us also define the following
norms for the above weighted spaces
\begin{align*}
	\|f\|_{\F} := \| \rho^{-1}_\F f\|_{ L^2(0, T; \Hdual)}, \quad 
	\|q\|_{\V} := \|\rho^{-1}_0 q\|_{L^2(0, T; L^2({\Gamma}))}.
\end{align*}
We have the following result in the same spirit as \cite[Theorem 3.2]{TT17}.
\begin{proposition}\label{Proposition-weighted}
	Let $T>0$.  For any given $f\in \F$ and for any $u_0 \in L^2(\Omega)$, there exists a control $q\in \V$ such that \eqref{lin_NKS_gen} admits a unique solution 
	$u\in \Y$ satisfying $u(T)=0.$
		Further, the solution and the control satisfy
	\begin{multline}\label{estimate-weighted}
		\left\|\frac{u}{\rho_0} \right\|_{C([0,T]; L^2(\Omega))\cap L^2(0,T;\mathcal H(\Omega)) }+\left\|\frac{\pa_t u}{\rho_0} \right\|_{L^2(0,T; \mathcal H'(\Omega))}  + \left\|\frac{q}{\rho_0} \right\|_{L^2((0,T;L^2(\Gamma))}\\
		\leq  Ce^{C(T+\frac{1}{T})}  \left(\|u_0\|_{ L^2(\Omega)} + \left\|{\frac{f}{\rho_\F}}\right\|_{L^2(0,T;\mathcal H'(\Omega))} \right),
	\end{multline}
	where the constant $C>0$ does not depend on $u_0$, $f$, $q$, $T$. 
\end{proposition}
\subsection{Proof of \Cref{th_nl_KS}}
	Let us consider the map $	\mathcal{F} \colon f \in \mathcal{S} \longmapsto -F(u) \in \mathcal{S},$
	where $(u, q)$ is the unique solution of \eqref{lin_NKS_gen}, and $F$ is defined by
	\[
	F(u) = -\frac{1}{2} |\nabla u|^2.
	\]
	We show that $\mathcal{F}$ is well-defined and that there exists $R > 0$ such that
	$\mathcal{F}(B(0, R)) \subset B(0, R),$
	where $B(0, R)$ denotes the closed ball in $\mathcal{S}$ of radius $R$.
	Using the Sobolev embedding $\mathcal{H}(\Omega) \subset L^\infty(\Omega)$ (for space dimensions 2 or 3) alongside an interpolation argument, we deduce that
	\[
	\|F(u)\|_{\Hdual} \leq C \|u\|_{H^1(\Omega)}^2 \leq C \|u\|_{H^2(\Omega)} \|u\|_{\Ltwo}.
	\]
	Therefore, applying the weighted estimate \eqref{estimate-weighted}, there exists a positive constant $C_1 > 0$ independent of $u_0$ and $f$ such that
	\begin{equation}\label{c}
		\|F(u)\|_{\mathcal{S}} \leq C_1 \left( \|u_0\|_{\Ltwo}^2 + \|f\|_{\mathcal{S}}^2 \right).
	\end{equation}
	{Assuming $\|u_0\|_{\Ltwo} < R$, $\|f\|_{\mathcal{S}} \le R$, and choosing $R < \frac{1}{2C_1}$ (with $C_1$ given by \eqref{c}), it easily follows that $\mathcal{F}(B(0, R)) \subset B(0, R)$.
	
	We now show that $\mathcal{F}$ is a contraction on $B(0, R)$. Let $f^{(1)}, f^{(2)} \in B(0, R)$, and let $(u^{(1)}, q^{(1)})$ and $(u^{(2)}, q^{(2)})$ be the corresponding solutions and controls of \eqref{lin_NKS_gen}. We set
	\begin{equation*}
		f := f^{(1)} - f^{(2)}, \quad u := u^{(1)} - u^{(2)}, \quad q := q^{(1)} - q^{(2)}.
	\end{equation*}
	Applying \eqref{estimate-weighted}, we have
	\begin{equation*}
		\left\| \frac{u}{\rho_0} \right\|_{C([0,T];\Ltwo)} + \left\| \frac{u}{\rho_0} \right\|_{L^2(0,T;\Hsp)} + \left\| \frac{q}{\rho_0} \right\|_{L^2(0,T;L^2(\Gamma))} \leq C_2 \|f\|_{\mathcal{S}}.
	\end{equation*}
	Furthermore, we observe that
	\begin{equation}\label{estimate-weighted2}
		\|F(u^{(1)}) - F(u^{(2)})\|_{\Hdual} \le C_3 \left( \|u^{(1)}\|_{H^1(\Omega)} + \|u^{(2)}\|_{H^1(\Omega)} \right) \|u\|_{H^1(\Omega)}.
	\end{equation}
	Combining the weighted estimates \eqref{estimate-weighted} applied to $ u^{(1)}$ and $ u^{(2)}$ with \eqref{estimate-weighted2}, we obtain
	\begin{equation}
		\|\mathcal{F}(f^{(1)}) - \mathcal{F}(f^{(2)})\|_{\mathcal{S}} = \|F(u^{(1)}) - F(u^{(2)})\|_{\mathcal{S}} \leq C_4 R \|f\|_{\mathcal{S}}.
	\end{equation}
	Further choosing $R < \frac{1}{2C_4}$, it follows that $\mathcal{F}$ is a contraction on $B(0,R)$. 
	
	By the Banach fixed-point theorem, $\mathcal{F}$ has a unique fixed point $f_0 \in B(0, R) \subset \mathcal{S}$. 
	In light of \Cref{Proposition-weighted}, for this fixed point $f_0 \in B(0,R)$ and initial data satisfying $\|u_0\|_{\Ltwo} \le R$, there exists a boundary control $q \in L^2(0,T; L^2(\Gamma))$ such that the corresponding state $u$ associated with $f_0$ satisfies the weighted energy estimate \eqref{estimate-weighted}. 
	Since the weight function satisfies $\lim_{t \to T^-} \rho_0(t) = 0$, the finiteness of the weighted norm forces the state to satisfy
	 $u(T) = 0$ in $\Omega$, which completes the proof of the local null-controllability for the nonlinear system \eqref{th_nl_KS}. }
	 
	 \vspace{0.5cm}
	
\par\noindent\textbf{Funding.} This work has received support from UNAM-DGAPA-PAPIIT grant IN117525 (Mexico). V. Hern\'andez-Santamar\'ia is supported by the program ``Estancias Posdoctorales por M\'exico para la Formaci\'on y Consolidaci\'on de las y los Investigadores por M\'exico'' of CONAHCYT (Mexico). He also received support from Project CBF2023-2024-116 of CONAHCYT and by UNAM-DGAPA-PAPIIT grants IA103826 and IN102925 (Mexico). Subrata Majumdar is supported by the UNAM Postdoctoral Program (POSDOC).

\section*{Acknowledgments}
The authors acknowledge Luz de Teresa for fruitful discussions. The authors also thank the anonymous referee and the Associate Editor for their valuable comments and suggestions, which helped improve the manuscript.

\bibliographystyle{amsalpha}
\bibliography{biblio}
\end{document}